\documentclass[a4paper,reqno,10pt]{amsart}
\makeatletter
\newcommand*{\rom}[1]{\expandafter\@slowromancap\romannumeral #1@}
\makeatother
\usepackage{epsf,graphicx,epsfig}
\usepackage{amscd}
\usepackage{amsmath,latexsym,amssymb,amsthm}
\usepackage[nospace,noadjust]{cite}
\usepackage{textcomp}
\usepackage{setspace,cite}
\usepackage{mathrsfs,amsmath}
\usepackage{chemarrow}
\usepackage{lscape,fancyhdr,fancybox}
\usepackage{stmaryrd}
\usepackage[all,cmtip]{xy}
\usepackage{tikz-cd}
\usepackage{cancel}
\usetikzlibrary{shapes,arrows,decorations.markings}
\usepackage{graphicx}

\usepackage{color}
\usepackage{bbm, dsfont}
\usepackage{xcolor}
\usepackage{url}
\usepackage{enumerate}
\usepackage[mathscr]{euscript}

  \theoremstyle{plain}
\swapnumbers
    \newtheorem{thm}{Theorem}[section]
    \newtheorem{prop}[thm]{Proposition}

    \newtheorem{corollary}[thm]{Corollary}
    
    \newtheorem{subsec}[thm]{}
\theoremstyle{definition}
    \newtheorem{defn}[thm]{Definition}
        \newtheorem{remark}[thm]{Remark}
\theoremstyle{remark}

\title{}
\author{}
\date{}
\usepackage{amssymb}

\usepackage{hyperref}
\hypersetup{
colorlinks,
citecolor=blue,
filecolor=black,
linkcolor=blue,
urlcolor=black
}

\begin{document}

\title[Nijenhuis Lie $2$-algebras]{Nijenhuis Lie $2$-algebras}

\author{Apurba Das}
\address{Department of Mathematics,
Indian Institute of Technology, Kharagpur 721302, West Bengal, India.}
\email{apurbadas348@gmail.com, apurbadas348@maths.iitkgp.ac.in}

\begin{abstract}
In this paper, we first introduce Nijenhuis Lie 2-algebras as the categorification of Nijenhuis Lie algebras. We prove that the category of Nijenhuis Lie 2-algebras is equivalent to the category of 2-term Nijenhuis $L_\infty$-algebras. Next, given a Nijenhuis Lie algebra, we introduce the notion of a 2-representation and show that the corresponding semidirect product inherits a Nijenhuis Lie 2-algebra structure. On the other hand, we consider a $2$-term representation up to homotopy of a Nijenhuis Lie algebra and obtain a $2$-term Nijenhuis $L_\infty$-algebra as the semidirect product. Finally, we show that the category of $2$-representations and the category of $2$-term representations up to homotopy of a Nijenhuis Lie algebra are equivalent.
\end{abstract}

\maketitle



\tableofcontents

\section{Introduction}
In mathematics, many higher structures, such as homotopy algebras, higher categories and `oidification' of known structures, have been extensively studied over the past few decades. These higher structures often provide the necessary background for describing higher Lie theory, higher gauge theory, and string theory \cite{chen,baez-hoff,baez-rogers}. Mathematically, higher structures arise when we increase the flexibility of an algebraic structure, and this can be done mainly in two ways: homotopification and categorification. The notion of homotopy algebras, also called sh algebras or $\infty$-algebras, is the homotopy-invariant extension of differential graded algebras (hence called the homotopification of algebras) and first appears in the topological study of loop spaces \cite{stas}. They are also used to describe D-branes in string theory \cite{kaji}. The corresponding Lie analogue of homotopy algebras, called $L_\infty$-algebras, can be treated as a model for Lie algebras where the Jacobi identity holds up to all higher homotopies \cite{lada-markl,lada-s}. This structure has wide applications in mathematics and mathematical physics, including its prominent role in the deformation quantization of Poisson manifolds \cite{kont}. On the other hand, the simplest higher structure that can be obtained as the categorification of a known structure is a $2$-vector space or a categorified vector space \cite{kap,baez-crans}. Further, endowing a Lie algebra structure on a $2$-vector space, one obtains the notion of a Lie $2$-algebra \cite{baez-crans}. This higher structure appears in string theory, higher symplectic geometry, the tetrahedron equation and other related fields \cite{baez-hoff,baez-rogers}. Baez and Crans \cite{baez-crans} have observed that the category of Lie $2$-algebras is equivalent to the category of $2$-term $L_\infty$-algebras (i.e., those $L_\infty$-algebras whose underlying graded vector space is concentrated in arities $0$ and $1$). The ongoing developments in homotopy algebras and related higher structures are reshaping modern research on symmetry, deformations, and higher-dimensional interactions.

\medskip

While algebras and their higher analogues have been studied over the past decades, various well-known operators on algebras and their higher variations have attracted significant attention in recent times. The simplest higher operator is given by an $\infty$-homomorphism between two $\infty$-algebras. In \cite{loday,doubek}, the authors first considered homotopy derivations on $\infty$-algebras and studied the associated operad. Very recently, Rota-Baxter operators and their higher analogues have been extensively studied in connection with cohomology theory, deformation theory, homotopy theory, pre-algebras and post-algebras (see \cite{sheng-survey} for a complete survey). Zhang and Liu \cite{zhang-liu} introduced Rota-Baxter operators on Lie $2$-algebras, and defined the category of Rota-Baxter Lie $2$-algebras. Extending the result of Baez and Crans, they showed that the category of Rota-Baxter Lie $2$-algebras is equivalent to the category of $2$-term Rota-Baxter $L_\infty$-algebras.

\medskip

An interesting operator that appears in the geometry of vector-valued differential forms \cite{fro}, linear deformation theory of algebraic structures \cite{koss}, integrable systems \cite{dorfman}, nonlinear evolution equations \cite{dorfman}, tensor hierarchies and bi-Hamiltonian systems \cite{gra-bi} is the `Nijenhuis operator'. This operator originates in the study of pseudo-complex manifolds \cite{nij}, and later appeared as a key object in the famous Newlander-Nirenberg theorem \cite{newlander}. For a Lie algebra $(\mathfrak{g}, [~, ~])$, a linear map $n : \mathfrak{g} \rightarrow \mathfrak{g}$ is said to be a Nijenhuis operator if it satisfies
\begin{align*}
    [n(x), n(y)] = n \big( [n(x), y] + [x, n(y)] - n ([x, y] ) \big), \text{ for } x, y \in \mathfrak{g}.
\end{align*}
Nijenhuis operators are closely related to Rota-Baxter operators, and induce NS-Lie algebras in the same way Rota-Baxter operators induce pre-algebras and post-algebras \cite{das-twisted}. A pair $( (\mathfrak{g}, [~, ~]), n)$ consisting of a Lie algebra $(\mathfrak{g}, [~, ~])$ together with a distinguished Nijenhuis operator $n$ is called a `Nijenhuis Lie algebra'. This structure appears naturally in the study of Lie bialgebras \cite{das,ravan,hou}, integrable systems and Poisson geometry \cite{azimi}. Nijenhuis Lie algebras become a valuable tool for investigating compatibility conditions, higher-order algebraic and geometric phenomena \cite{gra-bi,baishya}. Their applications extend to mathematical physics, differential geometry, and the theory of higher brackets, highlighting their significance in modern algebraic research \cite{koss,azimi}. Following the importance of this combined structure in various areas of mathematics and mathematical physics, the authors of \cite{das,song} have recently addressed the cohomology, deformations and homotopy theory of Nijenhuis Lie algebras. Among others, they discussed homotopy Nijenhuis operators briefly on $L_\infty$-algebras. Such homotopy operators provide a mechanism for generating a series of compatible higher brackets and controlling deformations within the homotopy framework. Characterizations of skeletal and strict $2$-term Nijenhuis $L_\infty$-algebras are also exhibited in connection with cohomology theory and crossed modules of Nijenhuis Lie algebras.

\medskip

This paper begins by considering Nijenhuis operators on a Lie $2$-algebra. For brevity, we call a Lie $2$-algebra together with a Nijenhuis operator simply a `Nijenhuis Lie $2$-algebra'. Therefore, a Nijenhuis Lie $2$-algebra can be regarded as the categorification of a Nijenhuis Lie algebra. We then introduce the category {\sf NijLie2} consisting of Nijenhuis Lie $2$-algebras. Subsequently, we also consider the category of $2$-term Nijenhuis $L_\infty$-algebras, denoted by ${\sf 2Term Nij L_\infty}$. Then we show that the category {\sf NijLie2} is equivalent to the category ${\sf 2Term Nij L_\infty}$ (cf. Theorem \ref{thmm-finall}). This extends the result of Baez and Crans to the context of Nijenhuis Lie algebras. As a consequence of our result, one can show that the category {\sf SNijLie2} of strict Nijenhuis Lie $2$-algebras (which is a subcategory of {\sf NijLie2}) is equivalent to the category ${\sf S2Term Nij L_\infty}$ of strict $2$-term Nijenhuis $L_\infty$-algebras (this is a subcategory of ${\sf 2Term Nij L_\infty}$).

\medskip

Like higher algebras extend classical algebras to the categorical and homotopical settings, the theory of higher representations extends classical representation theory to the same settings, providing a new framework for studying actions of (higher) algebraic structures on higher vector spaces and chain complexes. In this context, $2$-representations \cite{chu} and $2$-term representations up to homotopy \cite{abad-crainic,sheng-zhu} have emerged as fundamental objects, capturing symmetries that ordinary linear representations cannot adequately describe. These notions play a central role in the study of higher structures. In the last part of this paper, we initiate higher representations of a Nijenhuis Lie algebra. At first, we consider the notion of a $2$-representation or a categorified representation of a given Nijenhuis Lie algebra, and show that the corresponding semidirect product inherits a Nijenhuis Lie $2$-algebra structure. On the other hand, we formulate the concept of $2$-term representation up to homotopy of a Nijenhuis Lie algebra, and obtain a $2$-term Nijenhuis $L_\infty$-algebra as the corresponding semidirect product. Finally, we show that the category of $2$-representations and the category of $2$-term representations up to homotopy are equivalent (cf. Theorem \ref{thm-last}).

\medskip

The paper is organized as follows. In Section \ref{sec2}, we recall some necessary background on Lie $2$-algebras and $2$-term $L_\infty$-algebras. In Section \ref{sec3}, we introduce Nijenhuis operators on a given Lie $2$-algebra, and consider the category {\sf NijLie2} of Nijenhuis Lie $2$-algebras. In Section \ref{sec4}, we revise Nijenhuis operators on a $2$-term $L_\infty$-algebra, and then define the category ${\sf 2Term Nij L_\infty}$ consisting of $2$-term Nijenhuis $L_\infty$-algebras. Here we also prove that the category {\sf NijLie2} is equivalent to the category ${\sf 2Term Nij L_\infty}$. In Section \ref{sec5}, we introduce $2$-representations and $2$-term representations up to homotopy of a Nijenhuis Lie algebra. The paper concludes by showing that the category of $2$-representations and the category of $2$-term representations up to homotopy are equivalent.

\medskip

All vector spaces and (multi)linear maps are over a field ${\bf k}$ of characteristics $0$ unless specified otherwise.

\medskip

\section{Background on Lie $2$-algebras and $2$-term $L_\infty$-algebras}\label{sec2}

In this section, we recall some foundational background on Lie $2$-algebras and $2$-term $L_\infty$-algebras. We conclude by stating the result of Baez and Crans, which states that the category of Lie $2$-algebras and the category of $2$-term $L_\infty$-algebras are equivalent. For more details about the contents of this section, we suggest the references \cite{baez-crans,lada-markl,lada-s}.

\medskip

First, we begin with a $2$-vector space, which is an internal category in the category of vector spaces \cite{baez-crans}. Briefly, a {\bf $2$-vector space} is a category ${C}$ with a vector space of objects $C_0$ and a vector space of morphisms $C_1$ such that the source and target maps $s, t: C_1 \rightarrow C_0$, the object-inclusion map $i: C_0 \rightarrow C_1$, and the composition map $\circ: C_1 \times_{C_0} C_1  \rightarrow C_1$, are all linear maps. We often denote a $2$-vector space as above by $C= (C_1 \rightrightarrows C_0)$ when all the structure maps mentioned above are understood. Given a morphism $f \in C_1$ in a $2$-vector space $C = (C_1 \rightrightarrows C_0)$, we often consider its {\em arrow part} $\overrightarrow{f}$, which is defined by 
\begin{align*}
\overrightarrow{f} := f - i_{s(f)}.
\end{align*}


\medskip

\noindent Let ${C}= (C_1 \rightrightarrows C_0)$ and $C'= (C'_1 \rightrightarrows C'_0)$ be two $2$-vector spaces. Then a {\bf linear functor} $F: {C} \rightarrow C'$ is a functor internal to the category of vector spaces. That is, $F$ is given by a pair $F = (F_0, F_1)$ of linear maps $F_0 : C_0 \rightarrow C'_0$ and $F_1 : C_1 \rightarrow C'_1$ compatible with the structure maps of the $2$-vector spaces ${C}$ and $C'$. 

\medskip

\begin{remark}\label{rem-g}
Given $2$-vector spaces ${C}= (C_1 \rightrightarrows C_0)$ and $C'= (C'_1 \rightrightarrows C'_0)$, their {\em direct sum} is a new $2$-vector space $C \oplus C' := (C_1 \oplus C_1' \rightrightarrows C_0 \oplus C_0' )$ with the source map $s \oplus s'$, target map $t \oplus t'$, object-inclusion map $i \oplus i'$ and the composition map $\circ \oplus \circ'$. It is easy to see that the projection maps $\mathrm{pr}_{C_0} : C_0 \oplus C_0' \rightarrow C_0$ and $\mathrm{pr}_{C_1} : C_1 \oplus C_1' \rightarrow C_1$ (resp.  $\mathrm{pr}_{C'_0} : C_0 \oplus C_0' \rightarrow C'_0$ and $\mathrm{pr}_{C'_1} : C_1 \oplus C_1' \rightarrow C'_1$) yield a linear functor $\mathrm{pr}_C = ( \mathrm{pr}_{C_0}, \mathrm{pr}_{C_1}   )  : C \oplus C' \rightarrow C$ (resp. $\mathrm{pr}_{C'} = (\mathrm{pr}_{C'_0}, \mathrm{pr}_{C'_1} ) : C \oplus C' \rightarrow C'$) between $2$-vector spaces.
\end{remark}

\begin{defn}
 (i)   A {\bf Lie $2$-algebra} is a triple $(C, [~,~], \mathcal{J})$ consisting of a $2$-vector space $C = ({C}_1 \rightrightarrows {C}_0)$ equipped with an antisymmetric bilinear functor $[~,~]: C \times C \rightarrow C$ and a completely antisymmetric trilinear natural isomorphism (called the {\em Jacobiator})
    \begin{align*}
        \mathcal{J}_{x, y, z} : [[ x, y], z] \rightarrow [[x, z], y] + [x, [y, z]]
    \end{align*}
    satisfying the identity that can be expressed by the commutativity of the diagram:
    \[
\xymatrix{
  &  [[[x, y], z], t] \ar[rd]^1 \ar[ld]_{ [\mathcal{J}_{x,y,z}, 1] } &  \\
  [[[x, z], y], t] + [[ x, [y, z]], t] \ar[d]_{  \mathcal{J}_{[x,z], y, t} + \mathcal{J}_{x, [y, z], t} } &   &  [[[x, y], z], t] \ar[d]^{ \mathcal{J}_{[x, y], z, t}}  \\
\substack{  [[[x, z], t], y] + [[x, z], [y, t]]  \\ \quad  + [[x, t], [y, z]] + [x, [[y, z], t] ] } \ar[d]_{  [ \mathcal{J}_{x, z, t}, 1] + 1    } &  & [[[x,y],t],z] + [[x, y], [z, t]] \ar[d]^{ [ \mathcal{J}_{x, y, t}, 1] + 1}  \\
 \substack{  [[[x, t],z], y] + [[x, [z, t]], y] + [[x, z], [y, t]]  \\ + [[x, t], [y, z]] + [x, [[y, z], t] ] } \ar[rd]_{1 + [ 1 , \mathcal{J}_{y, z, t}]}  &  & \substack{  [[[x, t], y], z] + [[x, [y, t]], z] \\ + [[x, y], [z, t]] } \ar[ld]^{ \mathcal{J}_{[x, t], y, z} + \mathcal{J}_{x, [y, t], z} + \mathcal{J}_{x, y, [z, t]}} \\
  & \substack{  [[[ x, t], z], y] + [[x, t], [y, z]]  \\ + [[ x, z], [y, t]] + [ x, [[ y, t], z] ]    \\ + [[ x, [z, t]], y] + [ x, [y, [z, t]]].} & 
}
\]

\medskip

(ii) Let $(C, [~,~], \mathcal{J})$ and $(C', [~,~]', \mathcal{J}')$ be two Lie $2$-algebras. A {\bf homomorphism} of Lie $2$-algebras from the former one to the latter one is a linear functor $F= (F_0, F_1): C \rightarrow C'$ among the underlying $2$-vector spaces equipped with an antisymmetric bilinear natural transformation
\begin{align*}
    F_2 (x, y) : [F_0 (x), F_0 (y)]' \rightarrow F_0 [x, y]
\end{align*}
that makes the following diagram commutative:
\[
\xymatrix{
[[ F_0 (x), F_0 (y)]', F_0 (z)]' \ar[d]_{ [ F_2 (x, y), 1]'    } \ar[rrr]^{ \mathcal{J}'_{F_0 (x), F_0 (y), F_0 (z)} \qquad \qquad \qquad } & & & [ [ F_0 (x), F_0 (z)]', F_0 (y) ]' + [F_0 (x), [ F_0 (y), F_0 (z)]']' \ar[d]^{  [ F_2 (x, z), 1]' + [1, F_2 (y, z)]'    } \\
[ F_0 [x, y], F_0 (z)]' \ar[d]_{F_2 ([x, y], z)} & & & [ F_0 [x, z], F_0 (y)]' + [F_0 (x), F_0 [y, z]]' \ar[d]^{  F_2 ([x, z], y) + F_2 (x, [y, z])} \\
F_0 ([[x, y], z]) \ar[rrr]_{F_1 ( \mathcal{J}_{x, y, z})} & & & F_0 ([ [x, z], y] + [x, [y, z]]).
}
\]
\end{defn}

\medskip

\noindent A homomorphism as above is often denoted by the triple $(F_0, F_1, F_2)$. The collection of all Lie $2$-algebras and homomorphisms between them forms a category, denoted by ${\sf Lie2}.$ It has been observed in \cite{baez-crans} that Lie $2$-algebras are intimately related to $L_\infty$-algebras whose underlying graded vector space is concentrated in arities $0$ and $1$. Such $L_\infty$-algebras are known as $2$-term $L_\infty$-algebras.

\begin{defn}\label{defi-2term}
    (i) A {\bf $2$-term $L_\infty$-algebra} is a triple $\mathfrak{G} = (\mathfrak{g}_1 \xrightarrow{d} \mathfrak{g}_0, \llbracket ~, ~ \rrbracket , l_3)$ consisting of a $2$-term complex $\mathfrak{g}_1 \xrightarrow{d} \mathfrak{g}_0$ endowed with an antisymmetric bilinear map $\llbracket ~, ~ \rrbracket : \mathfrak{g}_i \times \mathfrak{g}_j \rightarrow \mathfrak{g}_{i+j}$ (for $0 \leq i+j \leq 1$) and a completely antisymmetric trilinear operation $l_3 : \mathfrak{g}_0 \times \mathfrak{g}_0 \times \mathfrak{g}_0 \rightarrow \mathfrak{g}_1$ satisfying the following list of identities: for $x, y, z, t \in \mathfrak{g}_0$ and $h, k \in \mathfrak{g}_1$,
    \begin{itemize}
        \item[(a)] $d \llbracket x, h \rrbracket = \llbracket x, d(h) \rrbracket$,
         \item[(b)] $ \llbracket  d(h), k \rrbracket =  \llbracket  h, d(k) \rrbracket,$
          \item[(c)] $d (l_3 (x, y, z)) = \llbracket x, \llbracket  y, z  \rrbracket   \rrbracket + \llbracket y, \llbracket  z, x  \rrbracket   \rrbracket + \llbracket z, \llbracket x,  y  \rrbracket   \rrbracket,$
           \item[(d)] $l_3 (x, y, d(h)) = \llbracket x, \llbracket  y, h  \rrbracket   \rrbracket + \llbracket y, \llbracket  h, x  \rrbracket   \rrbracket + \llbracket h, \llbracket x,  y  \rrbracket   \rrbracket,$
            \item[(e)] $\llbracket  x, l_3 (y, z, t) \rrbracket - \llbracket y, l_3 (x, z, t) \rrbracket +  \llbracket z, l_3 (x, y, t) \rrbracket - \llbracket t, l_3 (x, y, z) \rrbracket - l_3 ( \llbracket x, y \rrbracket, z, t ) \\
            \qquad +  l_3 ( \llbracket x, z \rrbracket, y, t ) -  l_3 ( \llbracket x, t  \rrbracket, y, z ) -  l_3 ( \llbracket y, z \rrbracket, x, t ) +  l_3 ( \llbracket y, t \rrbracket, x, z ) -  l_3 ( \llbracket z, t \rrbracket, x, y ) = 0.$
    \end{itemize}

    \medskip

    (ii) Let $\mathfrak{G} = (\mathfrak{g}_1 \xrightarrow{d} \mathfrak{g}_0, \llbracket ~, ~ \rrbracket , l_3)$ and $\mathfrak{G}' = (\mathfrak{g}'_1 \xrightarrow{d'} \mathfrak{g}'_0, \llbracket ~, ~ \rrbracket' , l'_3)$ be $2$-term $L_\infty$-algebras. An {\bf $L_\infty$-homomorphism} from $\mathfrak{G}$ to $\mathfrak{G}'$ consists of 
    \begin{itemize}
        \item[-] a chain map $ (\varphi_0, \varphi_1)$ among the underlying $2$-term complexes (i.e., $\varphi_0 : \mathfrak{g}_0 \rightarrow \mathfrak{g}'_0$ and $\varphi_1 : \mathfrak{g}_1 \rightarrow \mathfrak{g}'_1$ are linear maps satisfying $ \varphi_0 \circ d = d' \circ \varphi_1$),
        \item[-] an antisymmetric bilinear map $\varphi_2 : \mathfrak{g}_0 \times \mathfrak{g}_0 \rightarrow \mathfrak{g}'_1$ such that for all $x, y, z \in \mathfrak{g}_0$ and $h \in \mathfrak{g}_1$, 
    \end{itemize}
    \begin{align*}
        \varphi_0 \llbracket x, y \rrbracket - \llbracket \varphi_0 (x), \varphi_0 (y) \rrbracket' =~& d' (\varphi_2 (x, y)),\\
        \varphi_1 \llbracket x, h \rrbracket - \llbracket \varphi_0 (x), \varphi_1 (h) \rrbracket' =~& \varphi_2 (x, d(h)),\\
       \varphi_1 ( l_3 (x, y, z) ) - l_3' ( \varphi_0 (x), \varphi_0 (y), \varphi_0 (z)) =~& \varphi_2 (x, \llbracket y, z \rrbracket) + \varphi_2 (y, \llbracket z, x \rrbracket) + \varphi_2 (z, \llbracket x, y \rrbracket) \\
       &+ \llbracket \varphi_0 (x), \varphi_2 (y, z) \rrbracket' +  \llbracket \varphi_0 (y), \varphi_2 ( z, x) \rrbracket' +  \llbracket \varphi_0 (z), \varphi_2 (x, y) \rrbracket'.
    \end{align*}
\end{defn}

\noindent We shall denote an $L_\infty$-homomorphism as above by the triple $(\varphi_0, \varphi_1, \varphi_2)$. Then it can be shown that the collection of all $2$-term $L_\infty$-algebras and $L_\infty$-homomorphisms between them also forms a category. We denote this category by {\sf $2$TermL$_\infty$}. Then, the main result of \cite{baez-crans} is as follows.

\begin{thm}
The categories ${\sf Lie2}$ and {\sf $2$TermL$_\infty$} are equivalent. 
\end{thm}

This is a fundamental result, and its proof relies on some useful constructions. We will implicitly use them while proving our Theorem \ref{thmm-finall}.

\medskip

\section{Nijenhuis Lie 2-algebras}\label{sec3}

In this section, we first introduce Nijenhuis operators on Lie $2$-algebras. For brevity, we call a Lie $2$-algebra equipped with a Nijenhuis operator on it simply a Nijenhuis Lie $2$-algebra. Finally, we observe that the collection of all Nijenhuis Lie $2$-algebras with suitable homomorphisms between them forms a category. 


\begin{defn}\label{defi-31}
   Let $(C, [~, ~], \mathcal{J})$ be a Lie $2$-algebra. A {\bf Nijenhuis operator} on $C$ is a linear functor $N = (N_0, N_1) : C \rightarrow C$ with an antisymmetric bilinear natural isomorphism
   \begin{align}\label{cal-n}
       \mathcal{N}_{x, y} : [ N_0 (x), N_0 (y)] \rightarrow N_0 [ N_0 (x), y] + N_0 [x, N_0 (y)] - N_0^2 [x, y]
   \end{align}
    that makes the following diagram commutative:
    \begin{align}\label{big-diag}
    \xymatrix{
     & [ [ N_0 (x), N_0 (y) ] , N_0 (z) ] \ar[rd]^1 \ar[ld]_{\mathcal{J}_{ N_0 (x), N_0 (y), N_0 (z)} \qquad } & \\
  \substack{ [[ N_0 (x), N_0 (z)], N_0 (y) ] \\ + [N_0 (x), [ N_0 (y), N_0 (z)]] }  \ar[d]_{  [\mathcal{N}_{x, z}, 1] + [ 1, \mathcal{N}_{y, z}] } & & [ [ N_0 (x), N_0 (y) ] , N_0 (z) ] \ar[d]^{[ \mathcal{N}_{x, y}, 1 ]} \\
   \substack{ [ N_0 [  N_0 (x), z],  N_0 (y)] + [  N_0 [x,  N_0 (z)],  N_0 (y) ]  \\ - [ N_0^2 [x, z],  N_0 (y)] + [  N_0 (x),  N_0 [  N_0 (y), z]] \\ + [ N_0 (x),  N_0 [y,  N_0 (z)]] - [  N_0 (x),  N_0^2 [y, z]] } \ar[d]^{ \substack{    \mathcal{N}_{ [N_0 (x), z], y} + \mathcal{N}_{ [x, N_0 (z)], y }    - \mathcal{N}_{N_0 [x, z], y} \\ + \mathcal{N}_{x, [N_0(y), z]}     + \mathcal{N}_{x, [y, N_0 (z)]} - \mathcal{N}_{x, N_0 [y, z]}    }    } & &  \substack{ [N_0 [N_0 (x), y], N_0 (z)] \\ + [N_0 [x, N_0 (y)], N_0 (z)] \\ - [N_0^2 [x, y], N_0 (z)]} \ar[d]_{ \substack{  \mathcal{N}_{ [N_0 (x), y], z} + \mathcal{N}_{ [x, N_0 (y)], z} \\ - \mathcal{N}_{ N_0 [x, y], z} }   } \\
     A \ar[d]_{ \substack{ 1 + N_1 \mathcal{J}_{N_0 (x), z, N_0 (y) } \\ - N_1 \mathcal{N}_{[x, z], y} - N_1 \mathcal{N}_{x, [y, z]} }  } & &  A'  \ar[d]_{ \substack{  1 + N_1 \mathcal{J}_{N_0 (x), y, N_0 (z)} \\ - N_1^2 \mathcal{J}_{N_0 (x), y, z} + N_1 \mathcal{J}_{x, N_0 (y), N_0 (z)}  \\  - N_1^2 \mathcal{J}_{x, N_0 (y), z} - N_1 \mathcal{N}_{[x, y], z}   } } \\
     B \ar[rd]_{1 + N_1 [\mathcal{N}_{x, y}, 1 ]} & & B' \ar[ld]^{ \quad \qquad \substack{    1 + N_1 [\mathcal{N}_{x, z}, 1] + N_1 [1, \mathcal{N}_{y, z}] \\  -N_1^2 \mathcal{J}_{x, y, N_0 (z)} + N_1^3 \mathcal{J}_{x, y, z}   } } \\
      & C &
    }
    \end{align}
    where
    \begin{align*}
        A&= N_0 [ N_0 [ N_0 (x), z], y] + N_0 [[N_0 (x), z], N_0(y) ] - N_0^2 [[ N_0 (x), z], y] + N_0 [ N_0 [x, N_0 (z)], y] \\
         &~ ~+ N_0 [[ x, N_0 (z)], N_0 (y)] -N_0^2 [[x, N_0 (z)], y] - N_0 [ N_0^2 [x, z], y] - N_0 [ N_0 [x, z], N_0 (y)] + N_0^2 [ N_0 [x, z], y] \\
         &~ + N_0 [ N_0 (x), [ N_0 (y), z]] + N_0 [ x, N_0 [ N_0 (y), z]] - N_0^2 [x, [N_0 (y), z]] + N_0 [ N_0 (x), [y, N_0 (z)]] \\ &~+ N_0 [ x, N_0 [y, N_0 (z)]] - N_0^2 [x, [y, N_0 (z)]] - N_0 [ N_0 (x), N_0 [y, z]]- N_0 [x, N_0^2 [y, z]] + N_0^2 [x, N_0 [y,z]],
    \end{align*}
    \begin{align*}
        B &= N_0 [ N_0 [ N_0 (x), z], y] + N_0 [[ N_0 (x), N_0 (y)], z] - N_0^2 [ [ N_0 (x), z], y] + N_0 [ N_0 [x, N_0 (z)], y ] \\
        & ~ + N_0[[ x, N_0 (z)], N_0 (y)] - N_0^2 [[ x, N_0 (z)], y] - N_0 [ N_0^2 [x, z], y] - N_0^2 [[ x, z], N_0 (y)] + N_0^3 [[x, z], y] \\
        & ~ + N_0 [x, N_0 [ N_0 (y), z]] - N_0^2 [x, [ N_0 (y), z]] + N_0 [ N_0 (x), [y, N_0 (z)]] + N_0 [x, N_0 [y, N_0 (z)]] \\
        & ~ - N_0^2 [x, [y, N_0 (z)]] - N_0^2 [ N_0 (x), [y, z]] + N_0^3 [x, [y, z]] - N_0 [x, N_0^2 [y, z]],
    \end{align*}
    \begin{align*}
        A' &= N_0 [ N_0 [ N_0 (x), y], z] + N_0 [[ N_0 (x), y], N_0 (z)] - N_0^2 [[ N_0 (x), y], z] + N_0 [ N_0 [x, N_0 (y)], z] \\ &~+ N_0 [[ x, N_0 (y)], N_0 (z)] - N_0^2 [[x, N_0 (y)], z] - N_0 [ N_0^2 [x, y], z] - N_0 [ N_0 [x, y], N_0 (z)] + N_0^2 [ N_0 [x, y], z],
    \end{align*}
    \begin{align*}
        B' &= N_0 [ N_0 [ N_0 (x), y], z] + N_0 [[ N_0 (x), N_0 (z)], y] + N_0 [ N_0 (x), [y, N_0 (z)]] - N_0^2 [[ N_0 (x), z], y] \\
       & ~ ~ - N_0^2 [N_0 (x), [y, z]] + N_0 [ N_0 [x, N_0 (y)], z] + N_0 [[ x, N_0 (z)], N_0 (y)] + N_0 [ x, [ N_0 (y), N_0 (z)]] \\
       & ~ ~ - N_0^2 [[ x, z], N_0 (y)] - N_0^2 [x, [N_0 (y), z]] - N_0 [ N_0^2 [x, y], z] - N_0^2 [[ x, y], N_0 (z)] + N_0^3 [[x, y], z],
    \end{align*}
    \begin{align*}
        C &= N_0 [ N_0 [ N_0 (x), z], y] + N_0 [ N_0 [ N_0 (x), y], z] + N_0 [ N_0 [x, N_0 (y)], z] - N_0 [ N_0^2 [x, y], z] - N_0^2 [ [ N_0 (x), z], y] \\
        &+ N_0 [ N_0 [x, N_0 (z)], y ]  + N_0[[ x, N_0 (z)], N_0 (y)] - N_0^2 [[ x, N_0 (z)], y] - N_0 [ N_0^2 [x, z], y] - N_0^2 [[ x, z], N_0 (y)] \\ 
        &+ N_0^3 [[x, z], y]  + N_0 [x, N_0 [ N_0 (y), z]] - N_0^2 [x, [ N_0 (y), z]] + N_0 [ N_0 (x), [y, N_0 (z)]] + N_0 [x, N_0 [y, N_0 (z)]] \\
        & ~ - N_0^2 [x, [y, N_0 (z)]] - N_0^2 [ N_0 (x), [y, z]] + N_0^3 [x, [y, z]] - N_0 [x, N_0^2 [y, z]].
    \end{align*}
\end{defn}

\medskip

\noindent A Lie $2$-algebra equipped with a Nijenhuis operator is called a {\bf Nijenhuis Lie $2$-algebra}. We denote a Nijenhuis Lie $2$-algebra as above by the tuple $((C, [~, ~], \mathcal{J}), N, \mathcal{N})$.



\begin{defn}\label{defi-homo-rbl2}
Let $((C, [~, ~], \mathcal{J}), N, \mathcal{N})$ and $((C', [~, ~]', \mathcal{J}'), N', \mathcal{N}')$ be two Nijenhuis Lie $2$-algebras. Then a {\bf homomorphism} of Nijenhuis Lie $2$-algebras from the former one to the latter one is a homomorphism $(F_0, F_1, F_2): (C, [~, ~], \mathcal{J}) \rightarrow (C', [~, ~]', \mathcal{J}') $ of the underlying Lie $2$-algebras together with a linear natural transformation 
\begin{align}\label{f-thr}
    F_3 (x) : N_0' (F_0 (x)) \rightarrow F_0 (N_0 (x))
\end{align}
making the following diagram is commutative:
   \begin{align}\label{homo-nij}
    \xymatrix{
   [N_0' (F_0 (x)), N_0' (F_0 (y))]' \ar[r]^{\substack{\mathcal{N}'_{F_0 (x), F_0 (y)} \\ \\  }} \ar[d]_{[F_3 (x), F_3 (y)]'} &  \substack{N_0' [N_0' (F_0 (x)), F_0 (y)]' \\ + N_0' [ F_0 (x), N_0' (F_0 (y))]' \\ - {N'_0}^2 [F_0 (x), F_0 (y)]' } \ar[r]^{ \substack{ N_1' [F_3 (x),1]'  \\ + N_1' [1, F_3 (y)]' \\ - {N_1'}^2 F_2 (x, y)  } } & \substack{N_0'  [F_0 (N_0(x)), F_0 (y)]' \\ + N_0' [F_0 (x), F_0 (N_0 (y))]'  \\ - {N'_0}^2  F_0 [x, y] } \ar[d]_{ \substack{  N_1' F_2 (N_0 (x), y)   + N_1' F_2 (x, N_0 (y))   \\ - N_1' F_3 [x, y]  } } \\
 [F_0 (N_0 (x)), F_0 (N_0 (y))]'  \ar[d]_{ F_2 (N_0 (x), N_0 (y))  } &  & \substack{ N_0' ( F_0 [N_0 (x), y] + F_0 [x, N_0 (y)]) \\ - N_0' F_0 N_0 [x, y] }  \ar[d]_{ \substack{ F_3 [N_0 (x), y] + F_3 [x, N_0 (y)] \\ -  F_3 (N_0 [x, y] )} } \\
   F_0 [ N_0 (x), N_0 (y)] \ar[rr]_{F_1 (\mathcal{N}_{x, y}) } & & F_0 ( N_0 [N_0 (x), y] + N_0 [x, N_0 (y)] - N_0^2 [x, y]).
    }
    \end{align}

\end{defn}
\noindent A homomorphism as above is often abbreviated by the quadruple $F = (F_0, F_1, F_2, F_3)$ if no confusion arises. Then we have the following result.

\begin{prop}\label{prop-categ}
    The collection of all Nijenhuis Lie $2$-algebras, together with homomorphisms between them, forms a category. (We denote this category by ${\sf NijLie2}$).
\end{prop}

\begin{proof}
   Here, we only define the composition of homomorphisms and describe identity homomorphisms. Let 
   \begin{align*}
       F =~& (F_0 , F_1, F_2, F_3) : ((C, [~, ~], \mathcal{J}), N, \mathcal{N}) \rightarrow ((C', [~, ~]', \mathcal{J}'), N', \mathcal{N}'), \\
       G =~& (G_0 , G_1, G_2, G_3) : ((C', [~, ~]', \mathcal{J}'), N', \mathcal{N}') \rightarrow ((C'', [~, ~]'', \mathcal{J}''), N'', \mathcal{N}'')
   \end{align*}
   be two homomorphisms between Nijenhuis Lie $2$-algebras. Then their composition 
   \begin{align*}
       G \circ F : ((C, [~, ~], \mathcal{J}), N, \mathcal{N}) \rightarrow ((C'', [~, ~]'', \mathcal{J}''), N'', \mathcal{N}'')
   \end{align*}
   is defined to be the usual composition of the underlying linear functors (i.e., $ (G \circ F)_0 = G_0 \circ F_0$ and $ (G \circ F)_1 = G_1 \circ F_1$), and letting $(G \circ F)_2$ and $(G \circ F)_3$ be 
   \begin{align*}
        &(G \circ F)_2 (x, y) : [   (G \circ F)_0(x),  (G \circ F)_0(y)]'' \xrightarrow{ G_2{(F_0 (x), F_0 (y)})} G_0 [F_0 (x), F_0 (y)]' \xrightarrow{ G_1 ( F_2{(x, y)} )}  (G \circ F)_0 [x, y],\\
        & \quad \qquad (G \circ F)_3 (x) : N_0'' ( (G \circ F)_0 (x)) \xrightarrow{G_3 {(F_0 (x))}} G_0 ( N_0' (F_0 (x))) \xrightarrow{G_1 ( F_3(x) )} (G \circ F)_0 (N_0 (x)). 
   \end{align*}
   For any Nijenhuis Lie $2$-algebra $(C, N, \mathcal{N})$, the identity homomorphism $1_{(C, N, \mathcal{N})} : (C, N, \mathcal{N}) \rightarrow (C, N, \mathcal{N})$ is defined by taking the identity linear functor as its underlying functor, together with the identity natural transformations $(1_{(C, N, \mathcal{N})})_2$ and $(1_{(C, N, \mathcal{N})})_3$. With these structures, it is easy to see that ${\sf NijLie2}$ is a category.
\end{proof} 

\begin{remark}
A Nijenhuis Lie $2$-algebra $((C, [~, ~], \mathcal{J}), N, \mathcal{N})$ is said to be {\bf strict} if the underlying Lie $2$-algebra $(C, [~, ~], \mathcal{J})$ is strict in the sense that the Jacobiator $\mathcal{J}$ is the identity isomorphism, and $\mathcal{N}$ is also identity isomorphism. The collection of all strict Nijenhuis Lie $2$-algebras and homomorphisms between them is also a category, denoted by ${\sf SNijLie2}$. It is a subcategory of ${\sf NijLie2}$.
\end{remark}

\medskip

\section{2-term Nijenhuis $L_\infty$-algebras and the categorical equivalence}\label{sec4}

In this section, we briefly recall Nijenhuis operators on $2$-term $L_\infty$-algebras. We abbreviate a $2$-term $L_\infty$-algebra endowed with a distinguished Nijenhuis operator simply by a $2$-term Nijenhuis $L_\infty$-algebra. We prove that the collection of all $2$-term Nijenhuis $L_\infty$-algebras and suitable homomorphisms between them forms a category. Finally, this category is shown to be equivalent to the category ${\sf NijLie2}$ of Nijenhuis Lie $2$-algebras considered in the previous section.

\begin{defn}
    Let $\mathfrak{G} = (\mathfrak{g}_1 \xrightarrow{d} \mathfrak{g}_0, \llbracket ~, ~ \rrbracket , l_3)$ be a $2$-term $L_\infty$-algebra. Then a {\bf (homotopy) Nijenhuis operator} on $\mathfrak{G}$ is a triple $\mathfrak{N} = (n_0, n_1, n_2)$ consisting of

        \vspace{0.1cm}

    - linear maps $n_0 : \mathfrak{g}_0 \rightarrow \mathfrak{g}_0$ and $n_1 : \mathfrak{g}_1 \rightarrow \mathfrak{g}_1$ satisfying $n_0 \circ d = d \circ n_1$,

        \vspace{0.1cm}

    - an antisymmetric bilinear map $n_2 : \mathfrak{g}_0 \times \mathfrak{g}_0 \rightarrow \mathfrak{g}_1$

        \vspace{0.1cm}

  \noindent  such that for all $x, y, z \in \mathfrak{g}_0$ and $h \in \mathfrak{g}_1$, the following identities hold:
    \begin{align}
    n_0 ( \llbracket n_0 (x), y \rrbracket + \llbracket x, n_0 (y) \rrbracket - n_0 \llbracket x, y \rrbracket ) - \llbracket n_0 (x), n_0 (y) \rrbracket =~& d (n_2 (x, y)), \label{homo-r1} \\
     n_1 ( \llbracket n_0 (x), h \rrbracket + \llbracket x, n_1 (h) \rrbracket - n_1 \llbracket x, h \rrbracket ) - \llbracket n_0 (x), n_1 (h) \rrbracket =~& n_2 (x, d (h) ), \label{homo-r2}
    \end{align}
    \begin{align}\label{homo-r3}
       &\big\{ \llbracket n_0 (x), n_2 (y, z) \rrbracket + n_2 ( x,  \llbracket n_0 (y), z \rrbracket + \llbracket y, n_0 (z) \rrbracket - n_0 \llbracket y , z \rrbracket ) - n_1 \big(  \llbracket x, n_2 (y, z) \rrbracket  + n_2 (x, \llbracket y, z \rrbracket )   \big) \big\} + c.p. \nonumber \\
       & \qquad + l_3 ( n_0 (x), n_0 (y), n_0 (z)) - n_1 \big(   l_3 ( n_0 (x), n_0 (y), z) + l_3 ( n_0 (x), y, n_0 (z)) + l_3 (x, n_0 (y), n_0 (z)) \big) \nonumber  \\
       &\qquad \qquad \qquad + n_1^2 \big(  l_3 (n_0 (x), y, z) + l_3 (x, n_0 (y), z) + l_3 (x, y, n_0 (z))  \big) - n_1^3 ~ \! l_3 (x, y, z) = 0.
    \end{align}
\end{defn}


\noindent A {\bf $2$-term Nijenhuis $L_\infty$-algebra} is a pair $(\mathfrak{G}, \mathfrak{N}) = ( (\mathfrak{g}_1 \xrightarrow{d} \mathfrak{g}_0, \llbracket ~, ~ \rrbracket , l_3),   (n_0, n_1, n_2)  )$ of a $2$-term $L_\infty$-algebra equipped with a Nijenhuis operator on it.

\begin{defn}
    Let 
    \begin{align*}
    (\mathfrak{G}, \mathfrak{N}) = ( (\mathfrak{g}_1 \xrightarrow{d} \mathfrak{g}_0, \llbracket ~, ~ \rrbracket , l_3),   (n_0, n_1, n_2)  ) \quad  \text{ and } \quad (\mathfrak{G}', \mathfrak{N}') = ( (\mathfrak{g}'_1 \xrightarrow{d'} \mathfrak{g}'_0, \llbracket ~, ~ \rrbracket' , l'_3),   (n'_0, n'_1, n'_2)  )
    \end{align*}
\noindent be $2$-term Nijenhuis $L_\infty$-algebras. Then a {\bf Nijenhuis $L_\infty$-homomorphism} from the former one to the latter one is a tuple $\varphi = (\varphi_0, \varphi_1, \varphi_2, \varphi_3)$, where $(\varphi_0, \varphi_1, \varphi_2)$ is an $L_\infty$-homomorphism of the underlying $2$-term $L_\infty$-algebras from $(\mathfrak{g}_1 \xrightarrow{d} \mathfrak{g}_0, \llbracket ~, ~ \rrbracket, l_3)$ to $(\mathfrak{g}'_1 \xrightarrow{d'} \mathfrak{g}'_0, \llbracket ~, ~ \rrbracket', l'_3)$, and $\varphi_3: \mathfrak{g}_0 \rightarrow \mathfrak{g}_1'$ is a linear map satisfying
    \begin{align}
        \varphi_0 (n_0 (x)) - n_0' (\varphi_0 (x)) =~& d' (\varphi_3 (x)), \label{nijhomo1}\\
        \varphi_1 (n_1 (h)) - n_1' (\varphi_1 (h)) =~& \varphi_3 (d (h)), \label{nijhomo2}\\
        \varphi_1 (n_2 (x, y)) - n_2' ( \varphi_0 (x), \varphi_0 (y) ) =~& n_1' \big( \varphi_2 (n_0 (x), y) + \varphi_2 (x, n_0 (y)) - \varphi_3 ( \llbracket x, y \rrbracket )   \big) \nonumber \\
        + n_1'\big(  \llbracket \varphi_0 (x), \varphi_3 (y) \rrbracket' + \llbracket \varphi_3 (x), \varphi_0 (y) \rrbracket' - n_1' \varphi_2 (x, y)  \big) & + \varphi_3 \big(   \llbracket n_0 (x), y \rrbracket + \llbracket x, n_0 (y) \rrbracket - n_0 \llbracket x, y \rrbracket \big), \nonumber  \\
         - \varphi_2 ( n_0 (x), n_0 (y)) & - \llbracket \varphi_3 (x) , \varphi_3 (y) \rrbracket', \label{nijhomo3}
    \end{align}
    for all $x, y \in \mathfrak{g}_0$ and $h \in \mathfrak{g}_1$.
\end{defn}

With the above definitions, we have the following result.

\begin{prop}\label{prop-categ2}
    The collection of all $2$-term Nijenhuis $L_\infty$-algebras as objects and Nijenhuis $L_\infty$-homomorphisms between them is a category. (We denote this category by ${\sf 2Term Nij L_\infty}$).
\end{prop}

\begin{proof}
As before, here also we only define the composition of homomorphisms and describe the identity homomorphisms. Let 
\begin{align*}
    \varphi = (\varphi_0, \varphi_1, \varphi_2, \varphi_3) : (\mathfrak{G}, \mathfrak{N}) \rightarrow (\mathfrak{G}', \mathfrak{N}') \quad  \text{ and } \quad  \psi = (\psi_0, \psi_1, \psi_2, \psi_3) : (\mathfrak{G}', \mathfrak{N}') \rightarrow (\mathfrak{G}'', \mathfrak{N}'')
\end{align*}
be two Nijenhuis $L_\infty$-homomorphisms. We define their composition $\psi \circ \varphi : (\mathfrak{G}, \mathfrak{N})  \rightarrow (\mathfrak{G}'' , \mathfrak{N}'' ) $ by taking 
\begin{align*}
    (\psi \circ \varphi)_0 := \psi_0 \circ \varphi_0, \quad &(\psi \circ \varphi)_1 := \psi_1 \circ \varphi_1, \quad 
    (\psi \circ \varphi)_2 (x, y) := \psi_2 ( \varphi_0 (x), \varphi_0 (y)) + \psi_1 (\varphi_2 (x, y)), \\
    &\qquad (\psi \circ \varphi)_3 (x) := \psi_3 ( \varphi_0 (x)) + \psi_1 (\varphi_3 (x)),
\end{align*}
for $x, y \in \mathfrak{g}_0$. Next, for any $2$-term Nijenhuis $L_\infty$-algebra $(\mathfrak{G}, \mathfrak{N}) = ( (\mathfrak{g}_1 \xrightarrow{d} \mathfrak{g}_0, \llbracket ~, ~ \rrbracket , l_3),   (n_0, n_1, n_2)  )$, we define the identity Nijenhuis $L_\infty$-homomorphism $1_{ (\mathfrak{G}, \mathfrak{N})} : (\mathfrak{G}, \mathfrak{N}) \rightarrow (\mathfrak{G}, \mathfrak{N})$ by taking
\begin{align*}
    (1_{ (\mathfrak{G}, \mathfrak{N})})_0 = \mathrm{Id}_{\mathfrak{g}_0}, \quad  (1_{ (\mathfrak{G}, \mathfrak{N})})_1 = \mathrm{Id}_{\mathfrak{g}_1}, \quad (1_{ (\mathfrak{G}, \mathfrak{N})})_2 = 0 ~~~~ \text{ and } ~~~~ (1_{ (\mathfrak{G}, \mathfrak{N})})_3 = 0.
\end{align*}
With the above structures, it is easy to verify that ${\sf 2Term Nij L_\infty}$ is a category.
\end{proof}

\begin{remark}
A $2$-term Nijenhuis $L_\infty$-algebra  $(\mathfrak{G}, \mathfrak{N}) = ( (\mathfrak{g}_1 \xrightarrow{d} \mathfrak{g}_0, \llbracket ~, ~ \rrbracket , l_3),   (n_0, n_1, n_2)  )$ is said to be {\bf strict} if $l_3 = 0$ and $n_2 = 0$. The collection of all strict $2$-term Nijenhuis $L_\infty$-algebras and Nijenhuis $L_\infty$-homomorphisms between them is also a category, denoted by ${\sf S2Term Nij L_\infty}$. Then it is easy to see that ${\sf S2Term Nij L_\infty}$ is a subcategory of ${\sf 2Term Nij L_\infty}$.
\end{remark}

We are now ready to prove the main result of this paper.

\begin{thm}\label{thmm-finall}
    The categories $ {\sf Nij Lie2}$ and ${\sf 2Term Nij L_\infty}$ are equivalent.
\end{thm}

\begin{proof}
  Let $((C, [~, ~], \mathcal{J}), N, \mathcal{N})$ be a Nijenhuis Lie $2$-algebra. Since $(C, [~, ~], \mathcal{J})$ is a Lie $2$-algebra, it follows from \cite{baez-crans} that $(\mathfrak{g}_1 \xrightarrow{d} \mathfrak{g}_0, \llbracket ~, ~ \rrbracket , l_3)$ is a $2$-term $L_\infty$-algebra, where
  \begin{align}
      \mathfrak{g}_0 = C_0, \quad \mathfrak{g}_1 = \mathrm{ker} (s) \subset C_1, &\quad d(h) : = t(h), \quad \llbracket x, y \rrbracket := [x, y], \quad 
      \llbracket x, h \rrbracket = - \llbracket h , x \rrbracket := [i_x, h ] \nonumber \\ 
      ~~~ \text{ and } ~~~& l_3 (x, y, z) := \overrightarrow{\mathcal{J}_{x, y, z}}, \label{str-opr}
  \end{align}
  for $x, y, z \in \mathfrak{g}_0$ and $h \in \mathfrak{g}_1$. Next, from the given Nijenhuis operator $(N, \mathcal{N})$ on the Lie $2$-algebra, we define two linear maps $n_0 : \mathfrak{g}_0 \rightarrow \mathfrak{g}_0$ and $n_1 : \mathfrak{g}_1 \rightarrow \mathfrak{g}_1$, and an antisymmetric bilinear map $n_2 : \mathfrak{g}_0 \times \mathfrak{g}_0 \rightarrow \mathfrak{g}_1$ by
  \begin{align}
      n_0 (x) := N_0 (x), \quad n_1 (h) := N_1 (h) ~~~~ \text{ and } ~~~~ n_2 (x, y) := \overrightarrow{ \mathcal{N}_{x, y}}, \label{str-opr2}
  \end{align}
  for $x, y \in \mathfrak{g}_0$ and $h \in \mathfrak{g}_1$. For any $x, y \in \mathfrak{g}_0$, we observe that
  \begin{align*}
     & n_0 (  \llbracket n_0 (x), y \rrbracket + \llbracket x, n_0 (y) \rrbracket - n_0 \llbracket x, y \rrbracket  ) - \llbracket n_0 (x), n_0 (y) \rrbracket \\
      &= N_0 ( [ N_0 (x), y] + [x, N_0 (y)] - N_0 [x, y]) - [N_0 (x), N_0 (y)] \\
      &\stackrel{(\ref{cal-n})}= (t-s)~  \mathcal{N}_{x, y} = t ~ \overrightarrow{ \mathcal{N}_{x, y}} = d (n_2 (x, y)).
  \end{align*}
  Hence, the identity (\ref{homo-r1}) holds. Next, let $x \in \mathfrak{g}_0$ and $h \in \mathfrak{g}_1 = \mathrm{ker}(s)$. Then $h$ must be of the form $h = \overrightarrow{f}$, for some morphism $f: y \rightarrow z$. From the naturality of $\mathcal{N}$, we have
  \begin{align*}
      \mathcal{N}_{x, y} \circ  (N_1 [ i_{N_0 (x)}, f] + N_1 [i_x, N_1 (f)] - N_1^2 [i_x, f] )   =   [i_{N_0 (x)}, N_1 (f)] \circ \mathcal{N}_{x, z}.
  \end{align*}
  By considering the arrow parts of both sides of the above equation, we obtain
  \begin{align*}
        \overrightarrow{  \mathcal{N}_{x, y}} + \overrightarrow{  N_1 [ i_{N_0 (x)}, f] } + \overrightarrow{  N_1 [i_x, N_1 (f)] } - \overrightarrow{ N_1^2 [i_x, f] }  =   \overrightarrow{  [i_{N_0 (x)}, N_1 (f)] } + \overrightarrow{ \mathcal{N}_{x, z} },
  \end{align*}
  i.e.,
  \begin{align*}
      N_1 \big( [i_{N_0 (x)}, \overrightarrow{f}] + [i_x, N_1 (\overrightarrow{f}) ] - N_1 [i_x, \overrightarrow{f}] \big) - [ i_{N_0 (x)} , N_1 ( \overrightarrow{f})] = \overrightarrow{ \mathcal{N}_{x, z-y}}.
  \end{align*}
  This implies that
  \begin{align*}
      n_1 \big( \llbracket n_0 (x), \overrightarrow{f} \rrbracket + \llbracket x, n_1 (\overrightarrow{f}) \rrbracket - n_1 \llbracket x, \overrightarrow{f} \rrbracket  \big) - \llbracket n_0 (x), n_1 (\overrightarrow{f}) \rrbracket = n_2 (x, d (\overrightarrow{f}) ),
  \end{align*}
which verifies the identity (\ref{homo-r2}). Finally, the diagram given in (\ref{big-diag}) is equivalent to
\begin{align*}
 & \big\{ [ 1_{N_0 (x)}, \overrightarrow{\mathcal{N}_{y, z} }] + \overrightarrow{ \mathcal{N}_{x, [N_0 (y), z]+ [y, N_0 (z)]- N_0 [y, z]} } - N_1 \big(  [1_x, \overrightarrow{\mathcal{N}_{y, z}}] + \overrightarrow{\mathcal{N}_{x, [y, z]} }  \big) \big\} + c. p. \\
  &\qquad + \overrightarrow{ \mathcal{J}_{N_0 (x), N_0 (y), N_0 (z)}} - N_1 \big(  \overrightarrow{ \mathcal{J}_{N_0 (x), N_0 (y), z}} + \overrightarrow{ \mathcal{J}_{N_0 (x), y, N_0 (z)}} + \overrightarrow{ \mathcal{J}_{x, N_0 (y), N_0 (z)} }  \big) \\
  &\qquad \qquad + N_1^2 \big(  \overrightarrow{ \mathcal{J}_{N_0 (x), y, z}} + \overrightarrow{ \mathcal{J}_{x, N_0 (y), z}} + \overrightarrow{ \mathcal{J}_{x, y, N_0 (z)}}  \big) - N_1^3 \big(  \overrightarrow{ \mathcal{J}_{x, y, z}}  \big) =0.
\end{align*}
Therefore, the identity (\ref{homo-r3}) also holds. This shows that $(n_0, n_1, n_2)$ is a Nijenhuis operator on the $2$-term $L_\infty$-algebra $(\mathfrak{g}_1 \xrightarrow{d} \mathfrak{g}_0, \llbracket ~, ~ \rrbracket , l_3)$ constructed above. In other words, $((\mathfrak{g}_1 \xrightarrow{d} \mathfrak{g}_0, \llbracket ~, ~ \rrbracket , l_3) , (n_0, n_1, n_2))$ is a $2$-term Nijenhuis $L_\infty$-algebra.

Next, let $F = (F_0, F_1, F_2, F_3) : ((C, [~, ~], \mathcal{J}), N, \mathcal{N}) \rightarrow ((C', [~, ~]', \mathcal{J}'), N', \mathcal{N}')$ be a homomorphism between Nijenhuis Lie $2$-algebras. Since $(F_0, F_1, F_2) : (C, [~, ~], \mathcal{J}) \rightarrow (C', [~, ~]', \mathcal{J}')$ is a homomorphism between the underlying Lie $2$-algebras, it follows from \cite{baez-crans} that the triple
\begin{align*}
     (\varphi_0, \varphi_1, \varphi_2) : (\mathfrak{g}_1 \xrightarrow{d} \mathfrak{g}_0, \llbracket ~, ~ \rrbracket , l_3) \rightarrow (\mathfrak{g}'_1 \xrightarrow{d'} \mathfrak{g}'_0, \llbracket ~, ~ \rrbracket' , l'_3)
\end{align*}
is an $L_\infty$-homomorphism between the corresponding $2$-term $L_\infty$-algebras, where
\begin{align*}
    \varphi_0 (x) := F_0 (x), \quad \varphi_1 (h) := F_1 |_{\mathrm{ker} (s)} (h) ~~~~ \text{ and } ~~~~ \varphi_2 (x, y) := \overrightarrow{ F_2 (x, y)},
\end{align*}
for $x, y \in \mathfrak{g}_0$ and $h \in \mathfrak{g}_1$. Next, we define a map $\varphi_3 : \mathfrak{g}_0 \rightarrow \mathfrak{g}_1'$ by
\begin{align*}
     \varphi_3 (x) := \overrightarrow{ F_3 (x)}, \text{ for } x \in \mathfrak{g}_0.
\end{align*}
For any $x \in \mathfrak{g}_0$, we see that
\begin{align*}
    \varphi_0 (n_0 (x)) - n_0' (\varphi_0 (x)) =~& F_0 (N_0 (x)) - N_0'(F_0 (x))\\
    \stackrel{(\ref{f-thr})}=~& (t' - s') F_3 (x) = t' ~ \! \overrightarrow{ F_3 (x)} = d' (\varphi_3 (x)).
\end{align*}
Next, take $h \in \mathfrak{g}_1 = \mathrm{ker} (s)$. As before, let $h = \overrightarrow{f}$, for some morphism $f: y \rightarrow z$. Note that the naturality of $F_3$ implies that
\begin{align*}
   F_3 (y) \circ F_1 (N_1 (f))  = N_1'(F_1 (f)) \circ F_3 (z) .
\end{align*}
By taking the arrow parts of both sides, we obtain
\begin{align*}
    \overrightarrow{ F_3 (y)} +  F_1 (N_1 (\overrightarrow{f}))  =  N_1'(F_1 (\overrightarrow{f})) + \overrightarrow{F_3 (z)} .
\end{align*}
This shows that
\begin{align*}
    F_1 (N_1 ( \overrightarrow{ f })) - N_1'(F_1 ( \overrightarrow{f})) = \overrightarrow{F_3 (z-y)}, \quad  \text{ i.e., ~} ~~~ ~ \varphi_1 (n_1 (\overrightarrow{f})) - n_1' (\varphi_1 ( \overrightarrow{f})) = \varphi_3 ( d ( \overrightarrow{f})).
\end{align*}
Hence, both the identities (\ref{nijhomo1}) and (\ref{nijhomo2}) hold. Finally, the diagram given in Definition \ref{defi-homo-rbl2} is equivalent to 
\begin{align*}
\big( \overrightarrow{F_3 [N_0 (x), y]} ~+~& \overrightarrow{F_3 [x, N_0 (y)]} - \overrightarrow{F_3 (N_0 [x, y])} \big) + \big(  N_1' \overrightarrow{F_2 (N_0 (x), y)} + N_1' \overrightarrow{F_2 (x, N_0 (y))} - N_1' \overrightarrow{ F_3 [x, y]}   \big) \\
&+ N_1' \big(  [ \overrightarrow{ F_3 (x)}, i_{F_0 (y)}]' + [ i_{F_0 (x)}, \overrightarrow{ F_3 (y)}]'- N_1' \overrightarrow{ F_2 (x, y) }  \big) + \overrightarrow{ \mathcal{N}'_{F_0 (x), F_0 (y)} } \\
&\qquad \qquad = \overrightarrow{ F_1 (\mathcal{N}_{x, y})} + \overrightarrow{ F_2 (N_0 (x), N_0 (y))} + [ \overrightarrow{F_3 (x)}, \overrightarrow{F_3 (y)}]', 
\end{align*}
which shows the identity (\ref{nijhomo3}). This concludes that 
\begin{align*}
    (\varphi_0, \varphi_1, \varphi_2, \varphi_3) : ((\mathfrak{g}_1 \xrightarrow{d} \mathfrak{g}_0, \llbracket ~, ~ \rrbracket , l_3) , (n_0, n_1, n_2)) \rightarrow ((\mathfrak{g}'_1 \xrightarrow{d'} \mathfrak{g}'_0, \llbracket ~, ~ \rrbracket' , l'_3) , (n'_0, n'_1, n'_2))
\end{align*}
is a Nijenhuis $L_\infty$-homomorphism. Finally, the above construction preserves the identity homomorphisms and composition of homomorphisms. Thus, we obtain a functor ${\sf S}: {\sf NijLie2} \rightarrow {\sf S2Term Nij L_\infty}$.

\medskip

Our next aim is to construct a functor in the other direction. Suppose we start with a $2$-term Nijenhuis $L_\infty$-algebra $((\mathfrak{g}_1 \xrightarrow{d} \mathfrak{g}_0, \llbracket ~, ~ \rrbracket , l_3) , (n_0, n_1, n_2))$. Since $(\mathfrak{g}_1 \xrightarrow{d} \mathfrak{g}_0, \llbracket ~, ~ \rrbracket , l_3)$ is a $2$-term $L_\infty$-algebra, it follows from \cite{baez-crans} that $(C = (C_1 \rightrightarrows C_0), [~, ~], \mathcal{J})$ is a Lie $2$-algebra, where $C_0 = \mathfrak{g}_0$ and $C_1 = \mathfrak{g}_0 \oplus \mathfrak{g}_1$. The source, target, identity-assigning map and the composition map are respectively given by
\begin{align*}
    s (x, h) := x, ~~~ t (x, h) := x + d(h), ~~~i_x := (x, 0) ~~~ \text{ and }\\
    (x, h) \circ (y, k) = (x, h+ k ), \text{ when } t (x, h) = s (y, k).
\end{align*}
The antisymmetric bilinear functor $[~, ~ ] : C \times C \rightarrow C$ is given by 
\begin{align}\label{stru1}
    [x, y] := \llbracket x, y \rrbracket ~~~ \text{ and } ~~~ [(x, h), (y, k)] := ( \llbracket x, y \rrbracket, \llbracket x, k \rrbracket - \llbracket y, h \rrbracket + \llbracket dh , k \rrbracket ),
\end{align}
for objects $x, y \in C_0$ and morphisms $(x, h), (y, k) \in C_1$. Moreover, the Jacobiator for $C$ is given by
\begin{align}\label{stru2}
    \mathcal{J}_{x, y, z} := ( [[x, y], z] ~ \! , ~ \! l_3 (x, y, z)). 
\end{align}
Next, from the Nijenhuis operator $(n_0, n_1, n_2)$ on the $2$-term $L_\infty$-algebra $(\mathfrak{g}_1 \xrightarrow{d} \mathfrak{g}_0, \llbracket ~, ~ \rrbracket , l_3)$, we define a linear functor $N = (N_0, N_1) : C \rightarrow C$ by
\begin{align}\label{stru3}
    N_0 (x) := n_0 (x) ~~~~ \text{ and } ~~~~ N_1 (x, h) := (n_0 (x), n_1 (h)),
\end{align}
for $x \in C_0$ and $(x, h) \in C_1$. Moreover, we set a natural isomorphism 
\begin{align}\label{stru4}
   & \mathcal{N}_{x, y} : [ N_0 (x), N_0 (y)] \rightarrow N_0 [ N_0 (x), y] + N_0 [x, N_0 (y)] - N_0^2 [x, y] ~~~ \text{ by } \nonumber\\
    & \qquad \qquad \qquad \mathcal{N}_{x, y} : = (  [ N_0 (x), N_0 (y)] ~ \! , ~ \! n_2 (x, y)  ).
\end{align}
Then it turns out that the pair $(N, \mathcal{N})$ is a Nijenhuis operator on the Lie $2$-algebra $(C, [~, ~], \mathcal{J})$ constructed above. In other words, $((C, [~, ~], \mathcal{J}), N, \mathcal{N} )$ is a Nijenhuis Lie $2$-algebra.

Next, let $(\varphi_0, \varphi_1, \varphi_2, \varphi_3) : ((\mathfrak{g}_1 \xrightarrow{d} \mathfrak{g}_0, \llbracket ~, ~ \rrbracket , l_3) , (n_0, n_1, n_2)) \rightarrow ((\mathfrak{g}'_1 \xrightarrow{d'} \mathfrak{g}'_0, \llbracket ~, ~ \rrbracket' , l'_3) , (n'_0, n'_1, n'_2))$ be a Nijenhuis $L_\infty$-homomorphism. Since $(\varphi_0, \varphi_1, \varphi_2) : (\mathfrak{g}_1 \xrightarrow{d} \mathfrak{g}_0, \llbracket ~, ~ \rrbracket , l_3) \rightarrow (\mathfrak{g}'_1 \xrightarrow{d'} \mathfrak{g}'_0, \llbracket ~, ~ \rrbracket' , l'_3)$ is an $L_\infty$-homomorphism, it follows that the triple $(F_0, F_1, F_2)$ is a homomorphism between the corresponding Lie $2$-algebras, where
\begin{align*}
  &  F_0 : \mathfrak{g}_0 \rightarrow \mathfrak{g}_0' ~~~ \text{ is given by } ~~~ F_0 (x) := \varphi_0 (x),\\
  &  F_1 : \mathfrak{g}_0 \oplus \mathfrak{g}_1 \rightarrow \mathfrak{g}'_0 \oplus \mathfrak{g}'_1 ~~~ \text{ is given by } ~~~~ F_1 (x, h) := (\varphi_0 (x), \varphi_1 (h)), \\
  &  F_2 (x, y) : [F_0 (x), F_0 (y)]' \rightarrow F_0 [x, y]  ~~~ \text{ is given by } ~~~ F_2 (x, y) :=  \big(  [ \varphi_0 (x), \varphi_0 (y)]' ~ \! , ~ \! \varphi_2 (x, y)  \big).
\end{align*}
We also define a natural transformation $F_3 (x) : N_0'(F_0 (x)) \rightarrow F_0 (N_0 (x))$ by
\begin{align*}
    F_3 (x) := (   N_0'(F_0 (x)) ~ \! , ~ \! \varphi_3 (x) ).
\end{align*}
Note that the naturality of $F_3$ follows from the condition (\ref{nijhomo2}). The identity (\ref{nijhomo3}) ensures the commutativity of the diagram given in Definition \ref{defi-homo-rbl2}. Hence, the quadruple $(F_0, F_1, F_2, F_3) : (C, N, \mathcal{N}) \rightarrow (C', N', \mathcal{N}')$ is a homomorphism of Nijenhuis Lie $2$-algebras. Finally, the above construction preserves the identity homomorphisms and composition of homomorphisms. Hence, there is a functor ${\sf T}: {\sf 2Term Nij L_\infty} \rightarrow {\sf NijLie 2}$.

\medskip

It remains to show the existence of natural isomorphisms $\alpha : {\sf TS } \Rightarrow 1_{  {\sf NijLie 2} }$ and $\beta : {\sf ST} \Rightarrow 1_{ {\sf 2Term Nij L_\infty} }$.  Let $((C, [~,~], \mathcal{J}), N, \mathcal{N})$ be a Nijenhuis Lie $2$-algebra. By applying the functor ${\sf S}$, we get that
\begin{align*}
    {\sf S} (C, N, \mathcal{N}) = ((\mathfrak{g}_1 \xrightarrow{d} \mathfrak{g}_0, \llbracket ~, ~ \rrbracket , l_3) , (n_0, n_1, n_2))
\end{align*}
is a $2$-term Nijenhuis $L_\infty$-algebra, where all the structure operations are given in (\ref{str-opr}), (\ref{str-opr2}). Next, by applying the functor ${\sf T}$ to this $2$-term Nijenhuis $L_\infty$-algebra, we obtain a new Nijenhuis Lie $2$-algebra, say $((C', [~,~]', \mathcal{J}'), N', \mathcal{N}')$. Then we have
\begin{align*}
    &C_0' = \mathfrak{g}_0 = C_0, \quad C_1' = \mathfrak{g}_0 \oplus \mathfrak{g}_1 = C_0 \oplus \mathrm{ker}(s),\\
    &s' (x, h) = x, \quad t' (x, h) = x + d(h) = x + t (h), \quad i'_x = (x, 0), \\
    &[x, y]' = \llbracket x, y \rrbracket = [x, y], \quad [(x, h), (y, k)]' = ( [x, y], [i_x ,k] - [i_y, h] + [ i_{t(h)}, k]),\\
    &N_0' (x) = N_0 (x), \quad N_1' (x, h) = (N_0 (x), N_1 (h)) ~~~ \text{ and } ~~~ \mathcal{N}'_{x, y} = ( [N_0 (x), N_0 (y)] ~ \! , ~ \! n_2 (x, y)) = \mathcal{N}_{x, y},
\end{align*}
for $x, y \in C_0' = C_0$ and $(x, h), (y, k) \in C_1' = C_0 \oplus \mathrm{ker}(s)$. It is easy to see that
\begin{align*}
    \alpha_C := ( (\alpha_C)_0, (\alpha_C)_1, (\alpha_C)_2, (\alpha_C)_3 ) : ((C', [~,~]', \mathcal{J}'), N', \mathcal{N}') \rightarrow ((C, [~,~], \mathcal{J}), N, \mathcal{N})
\end{align*}
is an isomorphism of Nijenhuis Lie $2$-algebras, where $(\alpha_C)_0 (x) := x$, $(\alpha_C)_1 (x, h) := i_x + h$ and $(\alpha_C)_2, (\alpha_C)_3$ are defined to be the identity maps. This leads to a natural isomorphism $\alpha : {\sf TS } \Rightarrow 1_{\sf NijLie2}$.

On the other hand, let $(\mathfrak{G}, \mathfrak{N}) = ((\mathfrak{g}_1 \xrightarrow{d} \mathfrak{g}_0, \llbracket ~, ~ \rrbracket , l_3) , (n_0, n_1, n_2))$ be a $2$-term Nijenhuis $L_\infty$-algebra. After applying the functor ${\sf T}$, we assume that
\begin{align*}
    {\sf T} (\mathfrak{G}, \mathfrak{N}) = ((C, [~,~], \mathcal{J}), N, \mathcal{N})
\end{align*}
is a Nijenhuis Lie $2$-algebra, where the structure operations are given in (\ref{stru1})-(\ref{stru4}). By further applying the functor ${\sf S}$ to this Nijenhuis Lie $2$-algebra, we obtain a $2$-term Nijenhuis $L_\infty$-algebra, say $(\mathfrak{G}' , \mathfrak{N}')$. It is not hard to see that $(\mathfrak{G}' , \mathfrak{N}')$ is exactly same with $(\mathfrak{G}, \mathfrak{N})$. Hence, we have the identity isomorphism
\begin{align*}
    \beta_\mathfrak{G} = (     (\beta_\mathfrak{G})_0,  (\beta_\mathfrak{G})_1,  (\beta_\mathfrak{G})_2,  (\beta_\mathfrak{G})_3)  :  (\mathfrak{G}' , \mathfrak{N}') \rightarrow (\mathfrak{G} , \mathfrak{N})
\end{align*}
of $2$-term Nijenhuis $L_\infty$-algebras, where $(\beta_\mathfrak{G})_0 (x) = x$, $(\beta_\mathfrak{G})_1 (h) = h$ and $(\beta_\mathfrak{G})_2, (\beta_\mathfrak{G})_3$ are trivial maps. This yields a natural isomorphism $\beta : {\sf ST} \Rightarrow 1_{ {\sf 2Term Nij L_\infty} }$. Hence, the proof is done.
\end{proof}

\medskip

Inspired by the above result, one may also prove the following.

\begin{corollary}
      The category $ {\sf S NijLie2}$ of strict Nijenhuis Lie $2$-algebras and the category ${\sf S2Term Nij L_\infty}$ of strict $2$-term Nijenhuis $L_\infty$-algebras are equivalent.
\end{corollary}

\medskip

\section{2-representations and 2-term representations up to homotopy}\label{sec5}

In this section, we first introduce the notion of $2$-representation (i.e., categorified representation) of a Nijenhuis Lie algebra and show that the corresponding semidirect product inherits a Nijenhuis Lie $2$-algebra structure. Subsequently, we also consider $2$-term representation up to homotopy (i.e., $2$-term homotopified representation) of a Nijenhuis Lie algebra and find a $2$-term Nijenhuis $L_\infty$-algebra as the corresponding semidirect product structure. Our final result shows that, given a Nijenhuis Lie algebra, the category of $2$-representations and the category of $2$-term representations up to homotopy are equivalent.

\medskip

    Let $( (\mathfrak{g}, [~, ~ ]), n)$ be a Nijenhuis Lie algebra. First, recall that \cite{das,song} a {\bf representation} of $((\mathfrak{g}, [~, ~ ]), n)$ is a triple $(V, \rho, s)$, where $(V, \rho)$ is a representation of the Lie algebra $(\mathfrak{g}, [~, ~])$, and $s : V \rightarrow V$ is a linear map satisfying
    \begin{align*}
        \rho ({n(x)} ) s(v) = s \big(  \rho ({n(x)}) v + \rho (x) s(v) - s (\rho (x) v)  \big), \text{ for } x \in \mathfrak{g}, v \in V.
    \end{align*}
For a representation as above, the pair $((\mathfrak{g} \oplus V, [~, ~ ]_\ltimes), n \oplus s)$ is a new Nijenhuis Lie algebra, where
\begin{align*}
    [ (x, u), (y, v)]_\ltimes := ([x, y], \rho(x) v - \rho(y) u), \text{ for } (x, u), (y, v) \in \mathfrak{g} \oplus V,
\end{align*}
is the semidirect product Lie bracket on $\mathfrak{g} \oplus V$. Thus, the Nijenhuis Lie algebra constructed in this way is called the semidirect product of the given Nijenhuis Lie algebra $( (\mathfrak{g}, [~, ~ ]), n)$ with the representation $(V, \rho, s)$.

\medskip

We shall now consider $2$-representations of a Nijenhuis Lie algebra. Let $C$ be a $2$-vector space. Since $C$ is a linear category, we may consider the vector space $\mathrm{End}(C)$ of all linear endofunctors of $C$. With this notation, we have the following.

\begin{defn} Let $(\mathfrak{g}, [~, ~ ])$ be a Lie algebra.

\medskip

  (i)   Then a {\bf 2-representation} of $\mathfrak{g}$ is a triple $(C, \rho, \mathcal{R})$ consisting of a $2$-vector space $C$ with a linear assignment $\rho: \mathfrak{g} \rightarrow \mathrm{End}(C)$, and an antisymmetric (with respect to the elements of $\mathfrak{g}$) trilinear natural isomorphism
  \begin{align*}
      \mathcal{R}_{x, y} (v) : \rho ( [x, y] ) (v) \rightarrow ( \rho (x) \circ \rho(y) - \rho (y) \circ \rho(x) ) (v) = [\rho(x), \rho(y)] (v), \text{ for } x, y \in \mathfrak{g}, v \in C_0
  \end{align*}
  that makes the following diagram commutative:
  \begin{align}\label{2rep-diagram}
  \xymatrix{
   & \rho ( [[x, y], z])(v)  \ar[ld]_{= } \ar[rd]^{\mathcal{R}_{[x, y], z} (v)} &  \\
  \substack{ \rho ( [[x, z], y])(v) + \rho ([x, [y, z]])(v)  \\ } \ar[d]^{\mathcal{R}_{[x,z], y} (v) + \mathcal{R}_{x, [y, z]} (v) }  & & \substack{ [ \rho ([x, y]), \rho (z) ](v) \\ } \ar[d]^{ [ \mathcal{R}_{x, y} , 1_{\rho (z)}] (v) } \\
  \substack{ [ \rho ([x, z]), \rho (y)](v) + [ \rho (x), \rho ([y, z])](v) \\ } \ar[rd]_{ [\mathcal{R}_{x, z}, 1_{\rho (y)}] (v) + [ 1_{\rho (x)}, \mathcal{R}_{y, z}] (v) \qquad \qquad } & &  \substack{ [ [ \rho(x), \rho (y)], \rho (z)](v) \\ } \ar[ld]^{=} \\
   & \substack{[[ \rho (x), \rho (z)], \rho (y)](v) + [ \rho (x), [ \rho (y), \rho (z) ]] (v).\\  } & 
  }
  \end{align}

  \medskip

  (ii)  A {\bf homomorphism} of $2$-representations from $(C, \rho, \mathcal{R})$ to $(C', \rho', \mathcal{R}')$ is a linear functor $F = (F_0, F_1) : C \rightarrow C'$ among the underlying $2$-vector spaces with a bilinear natural transformation 
  \begin{align*}
      F_2 (x) (v) : \rho'(x) ( F_0 (v)) \rightarrow F_0 ( \rho (x) (v)) 
  \end{align*}
  making the below diagram commutative
  \begin{align}\label{2rep-mor-diag}
      \xymatrix{
      \rho' ([x, y]) (F_0 (v)) \ar[d]_{F_2 ([x, y]) (v)} \ar[rr]^{ \mathcal{R}'_{x, y} (F_0 (v)) } &  &  \substack{ (\rho' (x) \circ \rho' (y)) ( F_0 (v)) \\ - ( \rho' (y) \circ \rho' (x)) ( F_0 (v)) } \ar[r]^{ \substack{ {\rho ' (x)} ( F_2 (y) (v) ) \\ - {\rho ' (y)} ( F_2 (x) (v) ) } } &  \substack{ \rho' (x) (  F_0 (\rho (y)(v)) )  \\ -  \rho' (y) (  F_0 (\rho (x)(v)) )  }\ar[d]^{ F_2 (x) { (\rho (y) (v) )}  -  F_2 (y) {(\rho (x) (v))}   } \\
     F_0 ( \rho ([x, y]) (v)) \ar[rrr]_{F_1 ( \mathcal{R}_{x, y} (v) ) } & & & F_0 ( \rho (x) \circ \rho (y) - \rho (y) \circ \rho (x) ) (v).
      }
  \end{align}
  A homomorphism as above may be denoted by the triple $(F_0, F_1, F_2)$.
\end{defn}

\medskip

\begin{prop}\label{prop-compo}
    Let $(\mathfrak{g}, [~, ~ ])$ be a Lie algebra. Then the collection of all $2$-representations of $\mathfrak{g}$, together with homomorphisms between them, forms a category. (We denote this category by {\sf 2Rep}$(\mathfrak{g})$).
\end{prop}

\begin{proof}
    As before, we define the composition of homomorphisms and describe identity homomorphisms. Let
    \begin{align*}
        F = (F_0, F_1, F_2) : (C, \rho, \mathcal{R}) \rightarrow (C', \rho', \mathcal{R}') ~~~ \text{ and } ~~~ G = (G_0, G_1, G_2) : (C', \rho', \mathcal{R}') \rightarrow (C'', \rho'', \mathcal{R}'')
    \end{align*}
    be homomorphisms between $2$-representations. Then their composition $G \circ F : (C, \rho, \mathcal{R}) \rightarrow (C'', \rho'', \mathcal{R}'')$ is defined by taking $(G \circ F)_0 = G_0 \circ F_0$, $(G \circ F)_1 = G_1 \circ F_1$ and
    \begin{align*}
     (G \circ F)_2 (x) (v) : \rho'' (x) (( G_0 \circ F_0)(v)) \xrightarrow{ G_2 (x)  (F_0 (v))  } G_0 ( \rho' (x) (F_0 (v)) ) \xrightarrow{ G_1  (  F_2 (x)(v)) } (G_0 \circ F_0) ( \rho (x)(v)).
    \end{align*}
    On the other hand, given any $2$-representation $(C, \rho, \mathcal{R})$, the identity homomorphism $1_{ (C, \rho, \mathcal{R})  }$ is defined by taking $ ( 1_{ (C, \rho, \mathcal{R})  } )_0 = \mathrm{Id}_{C_0}$, $(1_{ (C, \rho, \mathcal{R})  })_1 = \mathrm{Id}_{C_1}$, together with the identity natural transformation $(1_{ (C, \rho, \mathcal{R})  })_2$. Then it is easy to see that the above structures make {\sf 2Rep}$(\mathfrak{g})$ into a category.
\end{proof}

Note that $2$-representations of a Lie algebra, and homomorphisms between them, have the following nice consequence. The verification is straightforward.

\begin{prop}\label{prop-lie2-semi}
  Let $(\mathfrak{g}, [~, ~])$ be a Lie algebra and $(C, \rho, \mathcal{R})$ be a $2$-representation of it. Then the triple
    \begin{align*}
       ( \mathfrak{g} \oplus C = (\mathfrak{g} \oplus C_1 \rightrightarrows \mathfrak{g} \oplus C_0), [~, ~]_\ltimes, \mathcal{J})
    \end{align*} 
    is a Lie $2$-algebra, where the antisymmetric bilinear functor $[~, ~ ]_\ltimes : (\mathfrak{g} \oplus C) \times (\mathfrak{g} \oplus C) \rightarrow \mathfrak{g} \oplus C$ and the Jacobiator $\mathcal{J}$ are respectively given by
    \begin{align*}
        &[ (x, u), (y, v) ]_\ltimes := ([x, y], \rho (x) v - \rho (y) u), \quad [ (x, h), (y, k)]_\ltimes := ([x, y], \rho(x)k - \rho (y) h), \\
        & \qquad \qquad \mathcal{J}_{(x, u), (y, v), (z, w)}:= \big( [[x,y],z] ~ \! \! , ~ \! \! \mathcal{R}_{x, y} (w)  + \widetilde{\mathcal{R}_{y, z}} (u) - \widetilde{\mathcal{R}_{x, z}} (v)  \big),
    \end{align*}
    for $(x, u), (y, v), (z, w) \in \mathfrak{g} \oplus C_0$ and $(x, h), (y, k) \in \mathfrak{g} \oplus C_1$. Here $\widetilde{\mathcal{R}_{y, z}} (u)$ is the composition
    \begin{align*}
        \rho (z) \rho (y) u \rightarrow \rho (z) \rho (y) u  - \rho ([y, z]) u + \rho ([y, z]) u \xrightarrow{1 + \mathcal{R}_{y, z} (u) } \rho (y) \rho(z) u - \rho ([y, z])u .
    \end{align*}
    (This Lie $2$-algebra is called the {\bf semidirect product} of the Lie algebra $\mathfrak{g}$ with the given $2$-representation $(C, \rho, \mathcal{R})$). Moreover, if $(C, \rho, \mathcal{R})$ and $(C', \rho', \mathcal{R}')$ are $2$-representations, and $(F_0, F_1, {F_2})$ is a homomorphism between them, then the triple
\begin{align*}
   ( \mathrm{Id} \oplus F_0, \mathrm{Id} \oplus F_1, \widetilde{ F_2}) :  ( \mathfrak{g} \oplus C, [~, ~]_\ltimes, \mathcal{J}) \rightarrow  ( \mathfrak{g} \oplus C', [~, ~]_\ltimes', \mathcal{J}') 
\end{align*}
is a homomorphism between the corresponding semidirect product Lie $2$-algebras, where 
\begin{align*}
    \widetilde{F_2} ((x, u), (y, v) ) := ( [x, y]~ \! \!  , ~\! \! F_2 (x) v - F_2 (y)u ), \text{ for } (x, u), (y, v) \in \mathfrak{g} \oplus C_0.
\end{align*}
\end{prop}

\begin{remark}\label{remark-l1}
    It follows from the above proposition that a $2$-representation of a Lie algebra $\mathfrak{g}$ gives rise to a Lie $2$-algebra structure, and a homomorphism between $2$-representations induces a homomorphism between the corresponding Lie $2$-algebras. Further, it is easy to verify that this construction preserves the identity homomorphisms and composition of homomorphisms, yielding a functor $\Phi: {\sf 2Rep} (\mathfrak{g}) \rightarrow {\sf Lie2}$.
\end{remark}

\medskip

\medskip

We shall now extend $2$-representations to the context of Nijenhuis Lie algebras. More precisely, we have the following.

\begin{defn}
    Let $( (\mathfrak{g}, [~, ~]), n)$ be a Nijenhuis Lie algebra.

    \medskip

 (i) Then a {\bf $2$-representation} of $( (\mathfrak{g}, [~, ~]), n)$ is a tuple $((C, \rho, \mathcal{R}), S, \Theta)$, where

    \vspace{0.1cm}

    - $(C, \rho, \mathcal{R})$ is a $2$-representation of the underlying Lie algebra $(\mathfrak{g}, [~, ~])$,

    \vspace{0.1cm}

    - $S= (S_0, S_1) : C \rightarrow C$ is a linear functor, with a bilinear natural isomorphism
    \begin{align*}
       \Theta_{x, v} : \rho (n(x) ) ~\! S_0 (v) \rightarrow S_0 \big( \rho (n (x)) v \big) + S_0 \big( \rho(x) S_0 (v) \big) - S_0^2 ( \rho (x) v)
    \end{align*}
    that makes a commutative diagram which can be obtained by expanding $\rho ( [n(x), n(y)]) S_0 (v)$ in two different ways similar to (\ref{big-diag}).

\medskip

  (ii) Let $((C, \rho, \mathcal{R}), S, \Theta)$ and $((C', \rho', \mathcal{R}'), S', \Theta')$ be $2$-representations. Then a {\bf homomorphism} between them is a quadruple $(F_0, F_1, F_2, F_3)$, where 

  \vspace{0.1cm}
  
  - the triple $(F_0, F_1, F_2) : (C, \rho, \mathcal{R}) \rightarrow (C', \rho', \mathcal{R}') $ is a homomorphism between $2$-representations of the Lie algebra $(\mathfrak{g}, [~, ~])$,

  \vspace{0.1cm}

  - $F_3 (v) : S_0' ( F_0 (v)) \rightarrow F_0 ( S_0 (v))$ is a linear natural transformation such that the diagram commutes

  \[
  \xymatrix{
 \rho'(n (x)) S_0' (F_0 (v)) \ar[d]_{\rho' (n (x)) F_3 (v) } \ar[r]^{\Theta'_{ x, F_0 (v)}} & \substack{ S'_0 ( \rho' (n (x)) F_0 (v)   )  \\ +  S'_0 (  \rho' (x) S_0' (F_0 (v)) ) \\ - {S'_0}^2 ( \rho' (x) F_0 (v) )} \ar[r]^{ \substack{ S_1' (F_2 (n(x)) v ) \\ + S_1' ( \rho' (x) F_3 (v) ) \\ - {S_1'}^2 (F_2 (x) v) }} & \substack{  S'_0 (  F_0 ( \rho (n(x)) v ) ) \\ + S_0' ( \rho' (x) F_0 (S_0 (v))) \\ - {S'_0}^2 F_0 (\rho (x) v) }\ar[d]^{\substack{ 1 + S_1'(F_2 (x) S_0 (v)) \\ - S_1'F_3 ( \rho (x) v) } } \\
\substack{ \rho' (n (x)) F_0 (S_0 (v)) \\ } \ar[d]_{ F_2 (n (x)) S_0 (v) } & & \substack{ S_0' ( F_0 ( \rho (n(x)) v ) ) \\   + S_0' (F_0 ( \rho (x) S_0 (v)))   \\  -S_0' F_0 S_0 (\rho (x) v)} \ar[d]^{ \substack{  F_3 ( \rho (n (x)) v ) + F_3 ( \rho (x) S_0 (v) ) \\  - F_3 (S_0 ( \rho (x) v ))}} \\
\substack{ F_0 (  \rho (n(x)) S_0 (v) ) \\ } \ar[rr]_{ F_1 (\Theta_{x, v})} & & \substack{ F_0 S_0 \big(   \rho (n(x)) v + \rho (x) S_0 (v) - S_0 ( \rho (x) v) \big) \\ }.
  }
  \]
\end{defn}

\begin{prop}
    Let $((\mathfrak{g}, [~, ~]), n)$ be a Nijenhuis Lie algebra. Then the collection of all $2$-representations of it, together with homomorphisms between them, forms a category. (We denote this category by {\sf 2Rep}$(\mathfrak{g}, n)$).
\end{prop}

\begin{proof}
    The composition of two homomorphisms $F = (F_0, F_1, F_2, F_3)$ and $G = (G_0, G_1, G_2, G_3)$ is defined by taking the composition of homomorphisms between the underlying $2$-representations of the Lie algebra $\mathfrak{g}$ (cf. Proposition \ref{prop-compo}), together with 
    \begin{align*}
        (G \circ F)_3 (v) : S_0'' ( (G_0 \circ F_0) (v)) \xrightarrow{G_3 (F_0 (v))} G_0 (S_0' (F_0 (v))) \xrightarrow{G_1 (F_3 (v))} (G_0 \circ F_0) S_0 (v).
    \end{align*}
    Moreover, the identity homomorphism $1_{(C, S, \Theta)} : ( (C, \rho, \mathcal{R}), S, \Theta) \rightarrow ( (C, \rho, \mathcal{R}), S, \Theta)$ is defined by taking the components as $(1_{(C, S, \Theta)})_0 = \mathrm{Id}_{C_0}$, $(1_{(C, S, \Theta)})_1 = \mathrm{Id}_{C_1}$, with the identity natural transformations $(1_{(C, S, \Theta)})_2$ and $(1_{(C, S, \Theta)})_3$.
\end{proof}

\begin{prop}\label{prop-im1}
    Let $((\mathfrak{g}, [~, ~]), n)$ be a Nijenhuis Lie algebra and $((C, \rho, \mathcal{R}), S, \Theta)$ be a $2$-representation of it. Then
    \begin{align*}
       ( (\mathfrak{g} \oplus C, [~, ~]_\ltimes, \mathcal{J}), N, \mathcal{N} )
    \end{align*}
    is a Nijenhuis Lie $2$-algebra, where the Lie $2$-algebra structure $(\mathfrak{g} \oplus C, [~, ~]_\ltimes, \mathcal{J})$ is described in Proposition \ref{prop-lie2-semi}, and 
    \begin{align*}
        N_0 (x, u) := (n(x), S_0 (u)), \quad N_1 (x, h) := (n (x), S_1 (h)) ~~~ \text{ and } ~~~ \mathcal{N}_{(x, u), (y, v)} := ( [n(x), n(y)] ~ \! \! , ~ \! \! \Theta_{x, v} - \Theta_{y, u} ),
    \end{align*}
    for $(x, u), (y, v) \in \mathfrak{g} \oplus C_0$ and $(x, h) \in \mathfrak{g} \oplus C_1$. Moreover, if $((C, \rho, \mathcal{R}), S, \Theta)$ and $((C', \rho', \mathcal{R}'), S', \Theta')$ are $2$-representations, and $(F_0, F_1, F_2, F_3)$ is a homomorphism between them, then the quadruple
    \begin{align*}
         ( \mathrm{Id} \oplus F_0, \mathrm{Id} \oplus F_1, \widetilde{ F_2}, \widetilde{ F_3} ) :  (( \mathfrak{g} \oplus C, [~, ~]_\ltimes, \mathcal{J}), N , \mathcal{N} ) \rightarrow ( ( \mathfrak{g} \oplus C', [~, ~]_\ltimes', \mathcal{J}'), N', \mathcal{N}' )
    \end{align*}
    is a homomorphism between the corresponding Nijenhuis Lie $2$-algebras, where
    \begin{align*}
        \widetilde{F_3} (x, u) : N_0' ( (\mathrm{Id} \oplus F_0) (x, u) ) \rightarrow (\mathrm{Id} \oplus F_0) N_0 (x, u) ~~ \text{ is given by } ~~  \widetilde{F_3} (x, u) = (n (x), F_3 (u)).
    \end{align*}
\end{prop}

\medskip

\medskip

We will now focus on $2$-term representations up to homotopy of a Nijenhuis Lie algebra, and consider the corresponding category. We begin with the following.

\begin{defn} 
    Let $(\mathfrak{g}, [~, ~])$ be a Lie algebra. 
    
    (i) (\cite{abad-crainic,sheng-zhu}) A {\bf 2-term representation up to homotopy} of $\mathfrak{g}$ is a tuple $(V_1 \xrightarrow{d} V_0, \rho_0, \rho_1, \nu)$ consisting of a $2$-term complex $ V_1 \xrightarrow{d} V_0$ equipped with two linear maps $\rho_0 : \mathfrak{g} \rightarrow \mathrm{End} (V_0)$ and $\rho_1 : \mathfrak{g} \rightarrow \mathrm{End}(V_1)$ that are compatible with $d$ in the sense that
    \begin{align*}
         d (\rho_1 (x) \xi ) = \rho_0 (x) d\xi, \text{ for } x \in \mathfrak{g}, ~ \! \xi \in V_1
    \end{align*}
    and an antisymmetric bilinear map $\nu : \mathfrak{g} \times \mathfrak{g} \rightarrow \mathrm{Hom} ( V_0, V_1)$ such that for  all $x, y, z \in \mathfrak{g}$,
    \begin{align}
        \rho_0 [x, y] - [ \rho_0 (x), \rho_0 (y)] =~& d  \circ \nu (x, y), \label{2term-g1}\\
        \rho_1 [x, y] - [\rho_1 (x) , \rho_1 (y)] =~& \nu (x, y) \circ d, \label{2term-g2}\\
        \big\{  \rho_1 (x) \circ \nu (y, z) - \nu (y, z) \circ \rho_0 (x)  \big\} + c.p. =~& \nu ([x, y], z) + c.p.~ \! . \label{2term-g3}
    \end{align}

    \medskip

    (ii) Let $(V_1 \xrightarrow{d} V_0, \rho_0, \rho_1, \nu)$ and $(V'_1 \xrightarrow{d'} V'_0, \rho'_0, \rho'_1, \nu')$ be $2$-term representations up to homotopy of $\mathfrak{g}$. Then a {\bf homomorphism} from the former one to the latter one is a triple $(\varphi_0, \varphi_1, \varphi_2)$, where

    \vspace{0.1cm}

    - $\varphi_0 : V_0 \rightarrow V_0'$ and $\varphi_1 : V_1 \rightarrow V_1'$ are linear maps satisfying $ \varphi_0 \circ d = d' \circ \varphi_1$,

    \vspace{0.1cm}

    - $\varphi_2 : \mathfrak{g} \rightarrow \mathrm{Hom} (V_0, V_1')$ is a linear map 

        \vspace{0.1cm}

\noindent such that for any $x, y \in \mathfrak{g}$, $v \in V_0$ and $\xi \in V_1$, the following list of identities hold:
\begin{align}
    \varphi_0 (\rho_0 (x) v) - \rho_0' (x) \varphi_0 (v) =~& d' (\varphi_2 (x) v), \\
    \varphi_1 (\rho_1 (x) \xi) - \rho_1' (x) \varphi_1 (\xi) =~& \varphi_2 (x) d\xi,\\
    \varphi_1 ( \nu (x, y) v) - \nu' (x, y) \varphi_0 (v) =~& \varphi_2 ([x, y]) v - \big\{  \varphi_2 (x) \rho_0 (y) v - \varphi_2 (y) \rho_0 (x) v  \big\} \nonumber \\&- \big\{ \rho_1' (x) \varphi_2 (y) v - \rho_1' (y) \varphi_2 (x) v \big\}.
\end{align}
\end{defn}

\medskip

Given a Lie algebra, we shall now construct the category of its $2$-term representations up to homotopy.

\begin{prop}\label{2termrepg}
    Let $(\mathfrak{g}, [~, ~ ])$ be a given Lie algebra. Then the collection of all $2$-term representations up to homotopy of $\mathfrak{g}$, together with homomorphisms between them, forms a category. (We denote this category by {\sf 2TermRep}$_\infty (\mathfrak{g})$).
\end{prop}

\begin{proof}
    Here, we only define the composition of homomorphisms and describe identity homomorphisms. Let
    \begin{align*}
       & \varphi = (\varphi_0, \varphi_1, \varphi_2) : (V_1 \xrightarrow{d} V_0, \rho_0, \rho_1, \nu) \rightarrow (V'_1 \xrightarrow{d'} V'_0, \rho'_0, \rho'_1, \nu'), \\
       & \psi = (\psi_0, \psi_1, \psi_2) : (V'_1 \xrightarrow{d'} V'_0, \rho'_0, \rho'_1, \nu') \rightarrow (V''_1 \xrightarrow{d''} V''_0, \rho''_0, \rho''_1, \nu'')
    \end{align*}
    be two homomorphisms between $2$-term representations up to homotopy of $\mathfrak{g}$. We define their composition
    \begin{align*}
        \psi \circ \varphi : (V_1 \xrightarrow{d} V_0, \rho_0, \rho_1, \nu) \rightarrow (V''_1 \xrightarrow{d''} V''_0, \rho''_0, \rho''_1, \nu'')
    \end{align*}
    by taking
    \begin{align*}
        (\psi \circ \varphi)_0 = \psi_0 \circ \varphi_0, \quad (\psi \circ \varphi)_1 = \psi_1 \circ \varphi_1 \quad \text{ and } \quad (\psi \circ \varphi)_2 (x) v = \psi_2 (x ) \varphi_0 (v) + \psi_1 ( \varphi_2 (x) v),
    \end{align*}
    for $x \in \mathfrak{g}$ and $v \in V_0$. Next, for any $2$-term representation up to homotopy $\mathcal{V} = ( V_1 \xrightarrow{d} V_0, \rho_0, \rho_1, \nu   )$, we define the identity homomorphism $1_\mathcal{V} : \mathcal{V} \rightarrow \mathcal{V}$ by taking $(1_\mathcal{V})_0 = \mathrm{Id}_{V_0}$, $(1_\mathcal{V})_1 = \mathrm{Id}_{V_1}$ and $(1_\mathcal{V})_2 = 0$. With this setup, one may easily verify that {\sf 2TermRep}$_\infty (\mathfrak{g})$ is a category.
\end{proof}

The following result shows that a $2$-term representation up to homotopy of a Lie algebra yields a $2$-term $L_\infty$-algebra. Moreover, the construction is functorial.

\begin{prop}\label{prop-semid}
    Let $(\mathfrak{g}, [~, ~])$ be a Lie algebra and $(V_1 \xrightarrow{d} V_0, \rho_0, \rho_1, \nu)$ be a $2$-term representation up to homotopy of $\mathfrak{g}$. Then the triple
    \begin{align*}
         (V_1 \xrightarrow{d} \mathfrak{g} \oplus V_0, \llbracket ~, ~ \rrbracket, l_3)
    \end{align*}
    is a $2$-term $L_\infty$-algebra, where for $(x, u), (y, v), (z, w) \in \mathfrak{g} \oplus V_0$ and $\xi, \eta \in V_1$,
    \begin{align*}
       & \llbracket (x, u), (y, v) \rrbracket := ( [x, y] ~ \! , ~ \!  \rho_0 (x) v - \rho_0 (y) u ),  \quad \llbracket (x, u), \xi \rrbracket = - \llbracket \xi , (x, u) \rrbracket := \rho_1 (x) \xi,  \quad \llbracket \xi, \eta \rrbracket := 0,\\
       & \qquad \qquad \qquad \qquad \qquad  l_3 ((x, u), (y, v), (z, w)) := - \nu (x, y) w + c.p. ~ \! .
    \end{align*}
(The $2$-term $L_\infty$-algebra constructed above is called the {\bf semidirect product} of the Lie algebra $\mathfrak{g}$ with the given $2$-term representation up to homotopy.)
     Next, let $(V_1 \xrightarrow{d} V_0, \rho_0, \rho_1, \nu)$ and $(V'_1 \xrightarrow{d'} V'_0, \rho'_0, \rho'_1, \nu')$ be $2$-term representations up to homotopy of a Lie algebra $\mathfrak{g}$, and $(\varphi_0, \varphi_1, \varphi_2)$ be a homomorphism between them. Then the triple
    \begin{align*}
         (\mathrm{Id} \oplus \varphi_0, \varphi_1, \widetilde{\varphi_2}) : (V_1 \xrightarrow{d} \mathfrak{g} \oplus V_0, \llbracket ~, ~ \rrbracket, l_3) \rightarrow (V'_1 \xrightarrow{d' } \mathfrak{g} \oplus V'_0, \llbracket ~, ~ \rrbracket', l'_3)
    \end{align*}
    is an $L_\infty$-homomorphism between the corresponding semidirect product $2$-term $L_\infty$-algebras, where the antisymmetric bilinear map $\widetilde{\varphi_2} : (\mathfrak{g} \oplus V_0 ) \times (\mathfrak{g} \oplus V_0 ) \rightarrow V_1'$ is given by
    \begin{align*}
        \widetilde{\varphi_2} ((x, u), (y, v) ) := \varphi_2 (x) v - \varphi_2 (y) u, \text{ for } (x, u), (y, v) \in \mathfrak{g} \oplus V_0.
    \end{align*}
\end{prop}

\begin{remark}\label{remark-l2}
    The above proposition shows that a homomorphism between $2$-term representations up to homotopy of a Lie algebra $\mathfrak{g}$ gives rise to an $L_\infty$-homomorphism between the corresponding $2$-term $L_\infty$-algebras. Moreover, it can be checked that this construction preserves the identity homomorphisms and composition of homomorphisms, and thus provides a functor $\Psi: {\sf 2TermRep}_\infty (\mathfrak{g}) \rightarrow  {\sf 2TermL}_\infty$.
\end{remark}

\medskip

We will now generalize the above constructions from Lie algebras to the context of Nijenhuis Lie algebras.

\begin{defn}\label{defi-2term-nijlie}
    Let $((\mathfrak{g}, [~, ~]), n)$ be a Nijenhuis Lie algebra. 

    \medskip
    
   \noindent (i) A {\bf $2$-term representation up to homotopy} of $((\mathfrak{g}, [~, ~]), n)$ is a tuple $ ((V_1 \xrightarrow{d} V_0, \rho_0, \rho_1, \nu), s_0, s_1, \theta)$, where

    \vspace{0.1cm}

    - $(V_1 \xrightarrow{d} V_0, \rho_0, \rho_1, \nu)$ is a $2$-term representation up to homotopy of the Lie algebra $(\mathfrak{g}, [~, ~]),$

       \vspace{0.1cm}

    - $s_0 : V_0 \rightarrow V_0$ and $s_1 : V_1 \rightarrow V_1$ are linear maps satisfying $s_0 \circ d = d \circ s_1$,

       \vspace{0.1cm}

    - $\theta : \mathfrak{g} \rightarrow \mathrm{Hom} (V_0, V_1)$ is a linear map 
    
  \noindent such that for all $x, y \in \mathfrak{g}$, $v \in V_0$ and $\xi \in V_1$, the following list of identities hold:
    \begin{align}
        &s_0 \big( \rho_0 (n(x)) v + \rho_0 (x) s_0 (v) - s_0 (    \rho_0 (x) v )   \big) - \rho_0 ( n (x) ) s_0 (v) = d (\theta (x) v ), \label{2term-n1} \\
        &s_1 \big(  \rho_1 (n (x)) \xi + \rho_1 (x) s_1 (\xi) - s_1 ( \rho_1 (x) \xi )   \big) - \rho_1 (n(x)) s_1 (\xi) = \theta (x) d(\xi), \label{2term-n2}
        \end{align}
        \begin{align}
        \rho_1 &(n(x)) \theta(y) v - \rho_1 (n(y)) \theta(x) v + \theta (x) \big(  \rho_0 (n(y))v + \rho_0 (y) s_0 (v) - s_0 ( \rho_0 (y) v)   \big) \nonumber \\
        &- \theta (y) \big(  \rho_0 (n(x))v + \rho_0 (x) s_0 (v) - s_0 ( \rho_0 (x) v)   \big) - \theta ([n(x), y] + [x, n(y)] - n [x, y]) v  \nonumber \\
        &\qquad - n_1 \big(   \rho_1 (x) \theta (y) v - \rho_1 (y) \theta (x) v + \theta (x) \rho_0 (y) v - \theta (y) \rho_0 (x) v - \theta ([x, y]) v \big)  \nonumber \\
        &\quad - \nu (n(x), n(y)) s_0 (v) + n_1 \big(  \nu (n(x), n(y)) v + \nu (n(x), y) s_0 (v) + \nu (x, n(y)) s_0 (v)  \big)  \nonumber \\
        &\qquad \qquad - n_1^2 \big(  \nu (x, y) s_0 (v) + \nu (n(x), y) v + \nu (x, n(y)) v  \big) + n_1^3 \big( \nu (x, y) v \big) = 0. \label{2term-n3}
    \end{align}

    \noindent (ii) Let $ ((V_1 \xrightarrow{d} V_0, \rho_0, \rho_1, \nu), s_0, s_1, \theta)$ and $ ((V'_1 \xrightarrow{d'} V'_0, \rho'_0, \rho'_1, \nu'), s'_0, s'_1, \theta')$ be $2$-term representations up to homotopy. Then a {\bf homomorphism} from the former one to the latter one is a quadruple $(\varphi_0, \varphi_1, \varphi_2, \varphi_3)$, where

    \vspace{0.1cm}

    - the triple $(\varphi_0, \varphi_1, \varphi_2)$ is a homomorphism of $2$-term representations up to homotopy (of the Lie algebra $\mathfrak{g}$) from $(V_1 \xrightarrow{d} V_0, \rho_0, \rho_1, \nu)$ to $(V'_1 \xrightarrow{d'} V'_0, \rho'_0, \rho'_1, \nu')$,

    \vspace{0.1cm}

    - $\varphi_3 : V_0 \rightarrow V_1'$ is a linear map

\noindent such that for all $x \in \mathfrak{g}$, $v \in V_0$ and $\xi \in V_1$, the following conditions hold:
\begin{align}
    \varphi_0 (s_0 (v)) - s_0' (\varphi_0 (v)) =~& d' (\varphi_3 (v)), \\
    \varphi_1 (s_1 (\xi)) - s_1' (\varphi_1 (\xi)) =~& \varphi_3 (d (\xi)), \\
    \varphi_1 ( \theta(x) v) - \theta'(x) \varphi_0 (v) =~& s_1' \big(  \varphi_2 ( n(x) ) v + \varphi_2 (x) s_0 (v) - \varphi_3 ( \rho_0 (x) v ) \big) \nonumber \\
    + s_1' \big(  \rho_1'(x) \varphi_3 (v) ~-~& s_1' (\varphi_2 (x) v) \big) + \varphi_3 \big(  \rho_0 (n(x))v + \rho_0 (x) s_0 (v) - s_0 ( \rho_0 (x) v) \big) - \varphi_2 (n(x)) s_0 (v).
\end{align}
\end{defn}

\medskip

The following result is a generalization of Proposition \ref{2termrepg} to the context of Nijenhuis Lie algebras.

\begin{prop}
    Let $( (\mathfrak{g}, [~, ~]), n)$ be a Nijenhuis Lie algebra. The collection of all $2$-term representations up to homotopy of it, together with homomorphisms between them, is a category. (This category is denoted by {\sf 2TermRep}$_\infty (\mathfrak{g},n)$).
\end{prop}
\begin{proof}
    Given two homomorphisms
    \begin{align*}
        &\varphi = (\varphi_0 , \varphi_1, \varphi_2, \varphi_3) :  ((V_1 \xrightarrow{d} V_0, \rho_0, \rho_1, \nu), s_0, s_1, \theta) \rightarrow  ((V'_1 \xrightarrow{d'} V'_0, \rho'_0, \rho'_1, \nu'), s'_0, s'_1, \theta'), \\
        &\psi= (\psi_0, \psi_1, \psi_2, \psi_3) : ((V'_1 \xrightarrow{d'} V'_0, \rho'_0, \rho'_1, \nu'), s'_0, s'_1, \theta') \rightarrow ((V''_1 \xrightarrow{d''} V''_0, \rho''_0, \rho''_1, \nu''), s''_0, s''_1, \theta'')
    \end{align*}
    between $2$-term representations up to homotopy, we define their composition $\psi \circ \varphi$ by setting
    \begin{align*}
         &(\psi \circ \varphi)_0 =  \psi_0 \circ \varphi_0, \quad (\psi \circ \varphi)_1 = \psi_1 \circ \varphi_1, \quad (\psi \circ \varphi)_2 (x) v = \psi_2 (x) \varphi_0 (v) + \psi_1 ( \varphi_2 (x) v) \\
         & \qquad \qquad \qquad \qquad \text{ and } ~ (\psi \circ \varphi)_3 (v) = \psi_1 ( \varphi_3 (v)) + \psi_3 (\varphi_0 (v)),
    \end{align*}
    for all $x \in \mathfrak{g}$ and $v \in V_0$. Moreover, if $\mathcal{V} =  ((V_1 \xrightarrow{d} V_0, \rho_0, \rho_1, \nu), s_0, s_1, \theta)$ is any $2$-term representation up to homotopy, then the identity homomorphism $1_\mathcal{V} : \mathcal{V} \rightarrow \mathcal{V}$ is defined by taking the components as 
    \begin{align*}
        (1_\mathcal{V})_0 = \mathrm{Id}_{V_0}, \quad (1_\mathcal{V})_1 = \mathrm{Id}_{V_1}, \quad (1_\mathcal{V})_2 = 0 ~~ \text{ and } ~~  (1_\mathcal{V})_3 = 0.
    \end{align*}
    With these structures, {\sf 2TermRep}$_\infty (\mathfrak{g},n)$ is a category.
\end{proof}

Next, we show that a $2$-term representation up to homotopy of a Nijenhuis Lie algebra provides us a $2$-term Nijenhuis $L_\infty$-algebra, and the construction is also functorial. This is indeed the motivation for considering Definition \ref{defi-2term-nijlie}.

\begin{prop}\label{prop-im2}
    Let $( (\mathfrak{g}, [~, ~]), n)$ be a Nijenhuis Lie algebra and $ ((V_1 \xrightarrow{d} V_0, \rho_0, \rho_1, \nu), s_0, s_1, \theta)$ be a $2$-term representation up to homotopy of it. Then 
    \begin{align*}
       ( \mathfrak{G}, \mathfrak{N}) := \big( (V_1 \xrightarrow{d} \mathfrak{g} \oplus V_0, \llbracket ~, ~ \rrbracket, l_3) , (n_0 = n \oplus s_0 ~ \! \! , ~ \! \! n_1 = s_1, n_2 ) \big)
    \end{align*}
    is a $2$-term Nijenhuis $L_\infty$-algebra, where the $2$-term $L_\infty$-algebra structure $(V_1 \xrightarrow{d} \mathfrak{g} \oplus V_0, \llbracket ~, ~ \rrbracket, l_3)$ is described in Proposition \ref{prop-semid}, and the antisymmetric bilinear map $n_2 : ( \mathfrak{g} \oplus V_0) \times ( \mathfrak{g} \oplus V_0) \rightarrow V_1$ is given by
    \begin{align*}
        n_2 ((x, u), (y, v)) := \theta (x) v - \theta (y) u, \text{ for } (x, u), (y, v) \in \mathfrak{g} \oplus V_0.
    \end{align*}
  \noindent  Moreover, if $ ((V_1 \xrightarrow{d} V_0, \rho_0, \rho_1, \nu), s_0, s_1, \theta)$ and $ ((V'_1 \xrightarrow{d'} V'_0, \rho'_0, \rho'_1, \nu'), s'_0, s'_1, \theta')$ are $2$-term representations up to homotopy, and $(\varphi_0, \varphi_1, \varphi_2, \varphi_3)$ is a homomorphism between them, then the quadruple
  \begin{align*}
      (\mathrm{Id} \oplus \varphi_0, \varphi_1, \widetilde{\varphi_2}, \widetilde{\varphi_3}) : ( \mathfrak{G}, \mathfrak{N}) \rightarrow ( \mathfrak{G}', \mathfrak{N}')
  \end{align*}
  is a Nijenhuis $L_\infty$-homomorphism between the corresponding $2$-term Nijenhuis $L_\infty$-algebras, where 
  \begin{align*}
      \widetilde{ \varphi_2} ((x,u), (y, v)) := \varphi_2 (x) v - \varphi_2 (y) u ~~ \text{ and } ~~ \widetilde{\varphi_3} (x,u) := \varphi_3 (u), \text{ for } (x, u), (y, v) \in \mathfrak{g} \oplus V_0.
  \end{align*}
\end{prop}

\begin{proof}
    Since $(V_1 \xrightarrow{d} V_0, \rho_0, \rho_1, \nu)$ is a $2$-term representation up to homotopy of the Lie algebra $\mathfrak{g}$, it follows from Proposition \ref{prop-semid} that $\mathfrak{G} = (V_1 \xrightarrow{d} \mathfrak{g} \oplus V_0, \llbracket ~, ~ \rrbracket, l_3) $ is a $2$-term $L_\infty$-algebra. Therefore, it remains to show that $\mathfrak{N} = (n_0 = n \oplus s_0 ~ \! \! , ~ \! \! n_1 = s_1, n_2 )$ is a Nijenhuis operator on $\mathfrak{G}$. First, from the condition $s_0 \circ d = d \circ s_1$, it is easy to see that the identity $n_0 \circ d = d \circ n_1$ holds. Next, for any $(x, u) , (y, v) \in \mathfrak{g} \oplus V_0$ and $\xi \in V_1$, we observe that
    \begin{align*}
       & n_0 \big(    \llbracket n_0 (x, u) , (y, v) \rrbracket + \llbracket (x, u), n_0 (y, v) \rrbracket - n_0 \llbracket (x, u), (y, v) \rrbracket \big) - \llbracket n_0 (x, u), n_0 (y, v) \rrbracket  \\
        &= \big( 0 ~ \! , ~ \!  s_0 \big( \rho_0 (n(x)) v + \rho_0 (x) s_0 (v) - s_0 (    \rho_0 (x) v )   \big) - \rho_0 ( n (x) ) s_0 (v)  \\
        & \qquad \qquad -   s_0 \big( \rho_0 (n(y)) u + \rho_0 (y) s_0 (u) - s_0 (    \rho_0 (y) u )   \big) + \rho_0 ( n (y) ) s_0 (u)   \big)\\
        &\stackrel{(\ref{2term-n1})}= \big(  0 ~ \! , ~ \! d ( \theta (x) v - \theta (y) u)  \big) = d \big( \theta (x) v - \theta (y) u   \big) = d \big(  n_2 ((x, u), (y, v))  \big)
    \end{align*}
    and 
     \begin{align*}
        &n_1 \big( \llbracket n_0 (x, u), \xi \rrbracket + \llbracket (x, u), n_1 (\xi) \rrbracket - n_1 \llbracket (x, u), \xi \rrbracket \big) - \llbracket n_0 (x, u), n_1 (\xi) \rrbracket   \\
        &= s_1 \big(  \rho_1 (n(x)) \xi + \rho_1 (x) s_1 (\xi) - s_1 (\rho_1 (x) \xi)    \big) - \rho_1 (n (x)) s_1 (\xi) \\
        &\stackrel{(\ref{2term-n2})}= \theta (x) d(\xi) = n_2 ((x, u), (0, d(\xi)) ) = n_2 ((x, u), d(\xi) ).
    \end{align*}
    Hence, the identities (\ref{homo-r1}) and (\ref{homo-r2}) hold. Finally, by using the condition (\ref{2term-n3}), one may also verify the identity (\ref{homo-r3}). This concludes that $\mathfrak{N} = (n_0 = n \oplus s_0 ~ \! \!  , ~ \! \! n_1 = s_1, n_2 )$ is a Nijenhuis operator on $\mathfrak{G}$, and hence $(\mathfrak{G}, \mathfrak{N})$ is a $2$-term Nijenhuis $L_\infty$-algebra.

    \medskip

    The last part is straightforward, and we leave it to the reader for verification.
\end{proof}

We will now prove the equivalence between the categories of $2$-representations and $2$-term representations up to homotopy of a given Nijenhuis Lie algebra. First, we do it for a given Lie algebra, and then extend it for a Nijenhuis Lie algebra.

\begin{thm}\label{theorem-mainth} Let $(\mathfrak{g}, [~, ~ ])$ be a Lie algebra. Then the categories {\sf 2Rep}$(\mathfrak{g})$ and {\sf 2TermRep}$_\infty (\mathfrak{g})$ are equivalent.
\end{thm}

\begin{proof}
    Let $(C, \rho, \mathcal{R})$ be a $2$-representation of the Lie algebra $\mathfrak{g}$. First, we consider the $2$-term complex $V_1 \xrightarrow{d} V_0$ associated to the $2$-vector space $C$, i.e.,
    \begin{align}\label{2rep-2term1}
        V_0 = C_0, \quad V_1 = \mathrm{ker}(s) \subset C_1 ~~~ \text{ and } ~~~ d (h) := t(h), \text{ for }h \in V_1.
    \end{align}
    We define linear maps $\rho_0 : \mathfrak{g} \rightarrow \mathrm{End}(V_0)$ and $\rho_1 : \mathfrak{g} \rightarrow \mathrm{End}(V_1)$ by 
    \begin{align}\label{2rep-2term2}
         \rho_0 (x) v := \rho(x) v ~~~~ \text{ and } ~~~~ \rho_1 (x) h := \rho(x) h, \text{ for } x \in \mathfrak{g}, v \in V_0 \text{ and } h \in V_1.
    \end{align}
    Then we have $d ( \rho_1 (x) h) = t (\rho(x) h ) = \rho (x) t (h) = \rho_0 (x) d (h)$, for any $h \in V_1$. We define an antisymmetric bilinear map $\nu : \mathfrak{g} \times \mathfrak{g} \rightarrow \mathrm{Hom} (V_0, V_1)$ by letting $\nu (x, y) v := - \overrightarrow{\mathcal{R}_{x, y} (v)}$, for $x, y \in \mathfrak{g}$ and $v \in V_0$. Then we observe that
    \begin{align*}
        \rho_0 ([x, y]) v - [ \rho_0 (x), \rho_0 (y)] v = \rho ([x, y] ) v - [\rho(x), \rho(y)] v = - (t-s) ~ \! \mathcal{R}_{x, y} (v) = - t ~ \! \overrightarrow{ \mathcal{R}_{x, y} (v) } = d (\nu (x, y) v).
    \end{align*}
    Next, let $h \in V_1$. Then there exists a morphism $f: u \rightarrow v$ such that $h = \overrightarrow{f}$. Hence, from the naturality of $\mathcal{R}$, we get
    \begin{align*}
        \overrightarrow{ \mathcal{R}_{x, y} (u)} + [ \rho (x), \rho (y)] \overrightarrow{f} = \rho ([x, y]) \overrightarrow{f} + \overrightarrow{ \mathcal{R}_{x, y} (v)}.
    \end{align*}
That is,
    \begin{align*}
        \rho ([x, y]) h - [ \rho (x), \rho (y)] h 
        =\overrightarrow{ \mathcal{R}_{x, y} (u)} - \overrightarrow{ \mathcal{R}_{x, y} (v)} 
        =~& \overrightarrow{ \mathcal{R}_{x, y} (u-v)} \\
        =~& \nu (x, y) (v -u) = \nu (x, y) d ( \overrightarrow{f}) = \nu (x, y) d(h).
    \end{align*}
Thus, the identities (\ref{2term-g1}) and (\ref{2term-g2}) hold. Finally, the diagram in (\ref{2rep-diagram}) is equivalent to
\begin{align*}
    \big\{  \overrightarrow{\mathcal{R}_{y, z} (\rho (x) v)} + \overrightarrow{ \mathcal{R}_{[y, z], x} (v) } - \rho (x) ~ \!( \overrightarrow{ \mathcal{R}_{y, z} (v)} ) \big\} + c.p. = 0,
\end{align*}
which verifies the identity (\ref{2term-g3}). Therefore, we get that $(V_1 \xrightarrow{d} V_0, \rho_0, \rho_1, \nu)$ is a $2$-term representation up to homotopy of the Lie algebra $\mathfrak{g}$.

\medskip

Next, let $(C, \rho, \mathcal{R})$ and $(C', \rho', \mathcal{R}')$ be $2$-representations of $\mathfrak{g}$, and $(F_0, F_1, F_2) : (C, \rho, \mathcal{R}) \rightarrow (C', \rho', \mathcal{R}')$ be a homomorphism between them. We define linear maps $\varphi_0 : V_0 \rightarrow V_0'$ and $\varphi_1 : V_1 \rightarrow V_1'$ by
\begin{align}\label{phi01}
    \varphi_0 (v) := F_0 (v) ~~~~ \text{ and } ~~~~ \varphi_1 (h) := F_1 |_{\mathrm{ker}(s)} (h), \text{ for } v \in V_0, ~ \! h \in V_1.
\end{align}
As $(F_0, F_1): C \rightarrow C'$ is a linear functor, we obviously have $\varphi_0 \circ d = d' \circ \varphi_1$. Next, we define a linear map $\varphi_2 : \mathfrak{g} \rightarrow \mathrm{Hom} (V_0, V_1')$ by
\begin{align}\label{phi22}
    \varphi_2 (x) v := \overrightarrow{ F_2 (x) (v) }, \text{ for } x \in \mathfrak{g}, v \in V_0.
\end{align}
Then we have
\begin{align*}
    \varphi_0 ( \rho_0 (x) v) - \rho_0' (x) \varphi_0 (v) =~& F_0 (\rho (x) v) - \rho'(x) F_0 (v) \\
    =~& (t' - s') F_2 (x) v = t' ~ \! \overrightarrow{ F_2 (x) v }
    = d' (\varphi_2 (x) v ).
\end{align*}
Next, let $h \in V_1$. As before, $h =\overrightarrow{f}$ for some morphism $f: u \rightarrow v$. Then
\begin{align*}
    \varphi_1 ( \rho_1 (x) h) - \rho_1' (x) \varphi_1 (h) =~& F_1 ( \rho(x) \overrightarrow{f}) - \rho_1' (x) F_1 (\overrightarrow{f}) \\
    =~& \overrightarrow{ F_2 (x) (v)} - \overrightarrow{ F_2 (x) (u)} \quad (\text{by the naturality of } F_2)\\
    =~& \varphi_2 (x) (v -u) = \varphi_2 (x) d (\overrightarrow{f}).
\end{align*}
Finally, the diagram (\ref{2rep-mor-diag}) implies that 
\begin{align*}
    \overrightarrow{ \mathcal{R}'_{x, y} ( F_0 (v))} - F_1 ( \overrightarrow{ \mathcal{R}_{x, y} (v) }) =~& \overrightarrow{ F_2 ([x, y]) (v) } - \big\{  \overrightarrow{ F_2 (x) ( \rho (y) v)} - \overrightarrow{ F_2 (y) ( \rho (x) v)}   \big\} \\
    &- \big\{   \rho (x) (\overrightarrow{ F_2 (y) v}) -   \rho (y) (\overrightarrow{ F_2 (x) v}) \big\}.
\end{align*}
This shows that $(\varphi_0, \varphi_1, \varphi_2) : (V_1 \xrightarrow{d} V_0, \rho_0, \rho_1, \nu) \rightarrow (V'_1 \xrightarrow{d'} V'_0, \rho'_0, \rho'_1, \nu')$ is a homomorphism between the corresponding $2$-term representations up to homotopy of the Lie algebra $\mathfrak{g}$. As a consequence, we arrive at a functor {\sf S} : {\sf 2Rep}$(\mathfrak{g}) \rightarrow ${\sf 2TermRep}$_\infty (\mathfrak{g})$.

\medskip

On the other hand, let $(V_1 \xrightarrow{ d} V_0, \rho_0, \rho_1, \nu)$ be a $2$-term representation up to homotopy of $\mathfrak{g}$. We first consider the $2$-vector space $C = (V_0 \oplus V_1 \rightrightarrows V_0)$ associated to the $2$-term complex $V_1 \xrightarrow{d} V_0$. Here, the source, target, object-inclusion and the composition map are respectively given by
\begin{align}\label{str-c1}
    s (v, h) = v, \quad t (v, h) = v + d(h), \quad i_v = (v, 0) ~ ~ \text{ and } ~ ~ (v, h) \circ (w, k) = (v, h+k), \text{ when } t (v, h) = s (w, k).
\end{align}
For any $x \in \mathfrak{g}$, we define a linear functor $\rho (x) = \big( \rho (x)_0, \rho(x)_1 \big) : C \rightarrow C$ by taking
\begin{align}\label{str-c2}
    \rho(x)_0 (v) := \rho_0 (x) v \quad \text{ and } \quad \rho (x)_1 (v, h) := \big( \rho_0 (x) v ~ \! , ~ \! \rho_1 (x) h   \big),
\end{align}
for $v \in V_0$ and $(v, h) \in V_0 \oplus V_1$. This defines a linear assignment $\rho: \mathfrak{g} \rightarrow \mathrm{End}(C),~ \!x \mapsto \rho (x)$. Finally, we define a trilinear natural isomorphism $\mathcal{R}_{x, y} (v) : \rho ([x, y])(v) \rightarrow [ \rho (x), \rho (y) ] (v) $, by
\begin{align}\label{str-c3}
    \mathcal{R}_{x, y} (v) := \big( \rho ([x, y]) v ~ \! , ~ \!  - \nu (x, y) v  \big).
\end{align}
Then it turns out that the diagram (\ref{2rep-diagram}) is commutative. Hence $(C = (V_0 \oplus V_1 \rightrightarrows V_0), \rho, \mathcal{R})$ is a $2$-representation of the Lie algebra $\mathfrak{g}$.

\medskip

Next, let $(V_1 \xrightarrow{d} V_0, \rho_0, \rho_1, \nu)$ and $(V'_1 \xrightarrow{d'} V'_0, \rho'_0, \rho'_1, \nu')$ be $2$-term representations up to homotopy of $\mathfrak{g}$, and let $(\varphi_0, \varphi_1, \varphi_2)$ be a homomorphism between them. We define linear maps $F_0 : V_0 \rightarrow V_0'$ and $F_1 : V_0 \oplus V_1 \rightarrow V_0' \oplus V_1'$ respectively by 
\begin{align}\label{map-f01}
F_0 (v) := \varphi_0 (v) ~~ \text{ and } ~~ F_1 (v, h)  := (\varphi_0 (v), \varphi_1 (h)), \text{ for } v \in V_0, (v, h) \in V_0 \oplus V_1. 
\end{align}
Then $(F_0, F_1) : (V_0 \oplus V_1 \rightrightarrows V_0) \rightarrow (V'_0 \oplus V'_1 \rightrightarrows V'_0)$ is a linear functor among the $2$-vector spaces. Next, we set a bilinear natural transformation $F_2 (x)(v) : \rho' (x) (F_0 (v) ) \rightarrow F_0 ( \rho (x)(v))$ by taking
\begin{align}\label{map-f2}
    F_2 (x) (v) := \big(  \rho' (x) (F_0 (v)) ~ \! , ~ \! \varphi_2 (x) v  \big), \text{ for } v \in V_0.
\end{align}
Then $( F_0 , F_1 , F_2 ) : (C = (V_0 \oplus V_1 \rightrightarrows V_0), \rho, \mathcal{R}) \rightarrow (C' = (V'_0 \oplus V'_1 \rightrightarrows V'_0), \rho', \mathcal{R}')$ is a homomorphism between the corresponding $2$-representations. As a result, one get a functor {\sf T}: {\sf 2TermRep}$_\infty (\mathfrak{g}) \rightarrow {\sf 2Rep} (\mathfrak{g})$.

\medskip

It remains to show the existence of natural isomorphisms $\alpha : {\sf TS} \Rightarrow 1_{{\sf 2Rep} (\mathfrak{g})}$ and $\beta: {\sf ST} \Rightarrow 1_{ {\sf 2TermRep}_\infty (\mathfrak{g})}$. Let $(C, \rho, \mathcal{R})$ be any $2$-representation of the Lie algebra $\mathfrak{g}$. After applying the functor {\sf S}, we suppose that ${\sf S} (C, \rho, \mathcal{R}) = (V_1 \xrightarrow{d} V_0, \rho_0, \rho_1, \nu)$ is the $2$-term representation up to homotopy of $\mathfrak{g}$. Then by applying the functor ${\sf T}$ to this structure, we obtain a $2$-representation, say $(C', \rho', \mathcal{R}')$. Then
\begin{align*}
    &C_0' = V_0 = C_0, \quad C_1' = V_0 \oplus V_1 = C_0 \oplus \mathrm{ker}(s),\\
    & s' (v, h) = v, \quad t' (v, h) = v + d(h) = v + t (h), \quad i'_v = (v, 0),\\
    &\rho' (x)_0 (v) = \rho_0 (x) v = \rho(x)v, \quad \rho'(x)_1 (v, h) = ( \rho_0 (x) v , \rho_1 (x) h ) = ( \rho (x) v , \rho(x) h), \\
    &\mathcal{R}'_{x, y} (v) = \big( \rho ([x, y]) v ~ \! , ~ \! - \nu (x, y) v  \big) = \big( \rho ([x,y]) v ~ \! , ~ \! \overrightarrow{ \mathcal{R}_{x, y} (v) } \big) = \mathcal{R}_{x, y} (v),
\end{align*}
for $v \in C_0'$ and $(v, h) \in C_1'$.
Hence $\alpha_C = ((\alpha_C)_0, (\alpha_C)_1, (\alpha_C)_2 ) : (C', \rho', \mathcal{R}') \rightarrow (C, \rho, \mathcal{R})$ is an isomorphism of $2$-representations, where $(\alpha_C)_0 (v) = v$, $(\alpha_C)_1 (v, h) = i_v + h $ and $(\alpha_C)_2 (x)$ is the identity natural transformation. This yields a natural isomorphism $\alpha : {\sf TS} \Rightarrow 1_{{\sf 2Rep} (\mathfrak{g})}$.

On the other side, suppose $\mathcal{V} = (V_1 \xrightarrow{d} V_0, \rho_0, \rho_1, \nu)$ is any $2$-term representation up to homotopy of $\mathfrak{g}$. By applying {\sf T}, we assume that ${\sf T} (\mathcal{V}) = (C, \rho, \mathcal{R})$ is a $2$-representation. After applying the functor ${\sf S}$ to this structure, we obtain a new $2$-term representation up to homotopy, say $\mathcal{V}' = (V'_1 \xrightarrow{d'} V'_0, \rho'_0, \rho'_1, \nu')$. Then it can be checked that $\mathcal{V}'$ is the same as $\mathcal{V}$. Hence, we have the identity isomorphism $\beta_\mathcal{V} = 1_\mathcal{V} : \mathcal{V}' \rightarrow \mathcal{V}$ between $2$-term representations up to homotopy. Thus, a natural isomorphism $\beta: {\sf ST} \Rightarrow 1_{ {\sf 2TermRep}_\infty (\mathfrak{g})}$ also exists. This concludes the proof.
\end{proof}

Let $(\mathfrak{g}, [~, ~])$ be a Lie algebra. Then the equivalence between the categories ${\sf 2Rep}(\mathfrak{g})$ and ${\sf 2TermRep}_\infty (\mathfrak{g})$ fits with the functors $\Phi$ and $\Psi$ given in Remarks \ref{remark-l1} and \ref{remark-l2}, respectively. Diagrammatically, this can be described as
\[
\xymatrix{
{\sf 2Rep(\mathfrak{g}) } \ar[r]^\Phi & {\sf Lie2} \\
{\sf 2TermRep_\infty (\mathfrak{g})} \ar@{<->}[u]^{\cong} \ar[r]_{ \quad \Psi} &  {\sf 2TermL_\infty } \ar@{<->}[u]_{\cong }.
}
\]

\medskip

Given a Lie algebra $(\mathfrak{g}, [~, ~ ])$, our Theorem \ref{theorem-mainth} says that the category of $2$-representations of $\mathfrak{g}$ and the category of $2$-term representations up to homotopy of $\mathfrak{g}$ are equivalent. This result can be extended to the context of Nijenhuis Lie algebras by taken care the Nijenhuis parts. Let $((\mathfrak{g}, [~, ~ ]), n)$ be a Nijenhuis Lie algebra. For any $2$-representation $((C, \rho, \mathcal{R}), S, \Theta)$ of it, the tuple $((V_1 \xrightarrow{d} V_0, \rho_0, \rho_1, \nu), s_0, s_1, \theta)$ is a $2$-term representation up to homotopy, where the structures of $(V_1 \xrightarrow{d} V_0, \rho_0, \rho_1, \nu)$ are described by (\ref{2rep-2term1}), (\ref{2rep-2term2}), and
\begin{align*}
    s_0 (v) := S_0 (v), \quad s_1 (h) := S_1 (h), \quad \theta (x) v := \overrightarrow{ \Theta_{x, v}},
\end{align*}
for $v \in V_0$, $h \in V_1$ and $x \in \mathfrak{g}$. Next, let $((C, \rho, \mathcal{R}), S, \Theta)$ and $((C', \rho', \mathcal{R}'), S', \Theta')$ be $2$-representations, and $(F_0, F_1, F_2, F_3)$ be a homomorphism between them. Then it is easy to see that the quadruple
\begin{align*}
    (\varphi_0, \varphi_1, \varphi_2, \varphi_3) : ((V_1 \xrightarrow{d} V_0, \rho_0, \rho_1, \nu), s_0, s_1, \theta) \rightarrow ((V'_1 \xrightarrow{d'} V'_0, \rho'_0, \rho'_1, \nu'), s'_0, s'_1, \theta')
\end{align*}
is a homomorphism between the corresponding $2$-term representations up to homotopy, where the maps $\varphi_0, \varphi_1, \varphi_2$ are given in (\ref{phi01}), (\ref{phi22}), and $\varphi_3 (v) : = \overrightarrow{F_3 (v)}$, for $v \in V_0$. This construction yields a functor $\overline{ {\sf S}} : {\sf 2Rep}(\mathfrak{g},n) \rightarrow {\sf 2TermRep_\infty }(\mathfrak{g}, n).$ On the other hand, suppose we start with a $2$-term representation up to homotopy, say $((V_1 \xrightarrow{d} V_0, \rho_0, \rho_1, \nu ), s_0, s_1, \theta)$. Then $ (( C = (V_0 \oplus V_1 \rightrightarrows V_0), \rho, \mathcal{R}), S, \Theta )$ is a $2$-representation, where the structures of $(C, \rho, \mathcal{R})$ are given by (\ref{str-c1})-(\ref{str-c3}), and for $v \in V_0$, $(v, h) \in V_0 \oplus V_1$ and $x \in \mathfrak{g}$,
\begin{align*}
    S_0 (v) := s_0 (v), \quad S_1 (v, h) = (s_0 (v), s_1 (h)), \quad \Theta_{x, v} := \big( \rho_0 (n(x)) s_0 (v) ~ \! , ~ \! \theta (x) v \big).
\end{align*}
Next, let $((V_1 \xrightarrow{d} V_0, \rho_0, \rho_1, \nu ), s_0, s_1, \theta)$ and $((V'_1 \xrightarrow{d'} V'_0, \rho'_0, \rho'_1, \nu' ), s'_0, s'_1, \theta')$ be $2$-term representations up to homotopy, and $(\varphi_0, \varphi_1, \varphi_2, \varphi_3)$ be a homomorphism between them. Then the quadruple 
\begin{align*}
    (F_0, F_1, F_2, F_3) : ((C, \rho, \mathcal{R}), S, \Theta) \rightarrow ((C', \rho', \mathcal{R}'), S', \Theta')
\end{align*}
is a homomorphism between the corresponding $2$-representations, where the components $F_0, F_1, F_2$ are given in (\ref{map-f01}), (\ref{map-f2}), and $F_3 (v) := \big(  s_0'(\varphi_0 (v)) ~ \! , ~ \! \varphi_3 (v) \big)$, for any $v \in V_0$. This construction then gives rise to a functor $\overline{\sf T} : {\sf 2TermRep_\infty }(\mathfrak{g}, n) \rightarrow {\sf 2Rep}(\mathfrak{g},n)$. Finally, similar to the proof of Theorem \ref{theorem-mainth}, we may obtain the natural isomorphisms $\overline{\alpha } : \overline{T} ~ \! \overline{S} \Rightarrow 1_{ {\sf 2Rep}(\mathfrak{g},n)}$ and $\overline{\beta} : \overline{S} ~ \! \overline{T} \Rightarrow 1_{{\sf 2TermRep_\infty }(\mathfrak{g}, n)}$. As a consequence, we get the following result.

\begin{thm}\label{thm-last}
    Let $((\mathfrak{g}, [~, ~ ]), n)$ be a Nijenhuis Lie algebra. Then the category {\sf 2Rep}$(\mathfrak{g}, n)$ is equivalent to the category {\sf 2TermRep}$_\infty (\mathfrak{g},n)$.
\end{thm}


 %

Let $((\mathfrak{g}, [~, ~]), n)$ be any Nijenhuis Lie algebra. In Proposition \ref{prop-im1}, we have seen that a $2$-representation of it gives rise to a Nijenhuis Lie $2$-algebra, and the construction is functorial. Hence, one obtains a functor $\overline{\Phi} : {\sf 2Rep}(\mathfrak{g},n) \rightarrow {\sf NijLie2}$ (which extends the functor $\Phi : {\sf 2Rep}(\mathfrak{g}) \rightarrow {\sf Lie2}$ to the context of Nijenhuis Lie algebras). On the other hand, in Proposition \ref{prop-im2}, we observed that a $2$-term representation up to homotopy of the Nijenhuis Lie algebra $((\mathfrak{g}, [~, ~]), n)$ provides a $2$-term Nijenhuis $L_\infty$-algebra via the semidirect product construction, and this is also functorial. As a result, one gets a functor $\overline{\Psi} : {\sf 2TermRep_\infty }(\mathfrak{g}, n) \rightarrow  {\sf 2TermNijL_\infty }$ (extending the functor $\Psi$ to this context). Then the categorical equivalence given in Theorem \ref{thm-last} fits into the following diagram

\[
\xymatrix{
{\sf 2Rep}(\mathfrak{g},n)  \ar[r]^{ \overline{\Phi}} & {\sf NijLie2} \\
{\sf 2TermRep_\infty }(\mathfrak{g}, n) \ar@{<->}[u]^{\cong} \ar[r]_{ \quad \overline{\Psi}} &  {\sf 2TermNijL_\infty } \ar@{<->}[u]_{\cong }.
}
\]

\medskip
    
\noindent  {\bf Acknowledgements.} The author would like to thank the Department of Mathematics, IIT Kharagpur for providing the beautiful academic atmosphere where the research has been carried out.

\medskip

\noindent {\bf Data Availability Statement.} Data sharing does not apply to this article as no new data were created or analyzed in this study.

\end{document}